\documentclass[nospthms, a4paper, final]{svjour3}
\usepackage{mathptmx}
\usepackage{comment}
\usepackage{tikz-cd}
\usepackage[bookmarks=true, linktocpage=true,
bookmarksnumbered=true, breaklinks=true,
pdfstartview=FitH, hyperfigures=false,
plainpages=false, naturalnames=true,
colorlinks=true, pagebackref=true,
pdfpagelabels]{hyperref}
\hypersetup{
	colorlinks,
	citecolor=blue,
	filecolor=blue,
	linkcolor=blue,
	urlcolor=blue
}

\usepackage{amsmath, amscd, amsbsy, amsthm, amsfonts, amssymb}
\usepackage{rotating}
\usepackage{dcolumn,multicol,url}
\newtheorem{thm}{THEOREM}[section]
\newtheorem{cor}[thm]{COROLLARY}
\newtheorem{lem}[thm]{LEMMA}
\newtheorem*{lem*}{LEMMA}

\theoremstyle{definition}

\numberwithin{equation}{section}

\newcommand{\awd}{AW_{\Delta}}
\newcommand{\taw}{\widetilde{AW}}
\newcommand{\tawd}{\widetilde{AW}_{\!\!\Delta}}

\newcommand{\A}{\mathcal{A}}

\def\ZZ{{\mathbb Z}}

\def\Zee{\mathbb{Z}}

\def\Id{\operatorname{Id}}

\def\cE{{\mathcal E}}

\def\cO{{\mathcal O}}
\def\cP{{\mathcal P}}

\def\cS{{\mathcal S}}

\def\FF{{\mathbb F}}
\def\ZZ{{\mathbb Z}}

\def\be{\begin{equation}}
\def\ee{\end{equation}}
\def\ba{\begin{eqneqnarray}}
\def\ea{\end{eqneqnarray}}
\def\tilde{\widetilde}

\def\e1{\epsilon}
\def\AAl{\mathcal{A}_{\lambda}}
\def\A0{\stackrel{\circ}{\AAl}}

\def\o1{\omega}
\def\01{\Omega}
\def\c1{\gamma}

\def\g1{\Sigma}

\def\l1{\Lambda}

\def\v1{\varphi}

\def\d1{\delta}

\def\f1{\frac}
\def\t1{\theta}
\def\b1{\beta}
\def\bar{\overline}
\def\bs{\begin{eqneqnarray*}}
\def\es{\end{eqneqnarray*}}
\def\m1{\Theta}
\def\w1{\wedge}

\def\tx{\tilde{x}}

\def\ee{\epsilon}

\begin{document}

\title{A Cochain Level Proof of Adem Relations in the Mod 2 Steenrod Algebra \thanks{A.M-M acknowledges financial support from Innosuisse grant \mbox{32875.1 IP-ICT - 1} and the Max Planck Institute for Mathematics in Bonn.}}

\author{Greg Brumfiel \and Anibal Medina-Mardones \and John Morgan}

\institute{G. Brumfiel \at
	Department of Mathematics, Stanford University, 450 Serra Mall, Palo Alto, CA 94305, USA \\
	\email{brumfiel@math.stanford.edu} \and
	A. Medina-Mardones \at
	Laboratory of Topology and Neuroscience, \'Ecole Polytechnique F\'ed\'erale de Lausanne, Station 8, CH-1015, Switzerland.\\
	\email{ammedmar@mpim-bonn.mpg.de} \and
	J. Morgan \at
	Department of Mathematics, Columbia University, 116th Street, New York City, NY 10027, USA \\
	\email{jmorgan@math.columbia.edu}
}

\authorrunning{G. Brumfiel, A. Medina-Mardones, J. Morgan}

\maketitle

\begin{abstract}
	In 1947, N.E. Steenrod defined the Steenrod Squares, which are mod 2 cohomology operations, using explicit cochain formulae for cup-$i$ products of cocycles. He later recast the construction in more general homological terms, using group homology and acyclic model methods, rather than explicit cochain formulae, to define mod $p$ operations for all primes $p$. Steenrod's student J. Adem applied the homological point of view to prove fundamental relations, known as the Adem relations, in the algebra of cohomology operations generated by the Steenrod operations. In this paper we give a proof of the mod 2 Adem relations at the cochain level. Specifically, given a mod 2 cocycle, we produce explicit cochain formulae whose coboundaries are the Adem relations among compositions of Steenrod Squares applied to the cocycle, using Steenrod's original cochain definition of the Square operations.
\end{abstract}

\section {Introduction}

The primary goal of this paper is to produce explicit coboundary formulae that yield the Adem relations between compositions of Steenrod Squares.
There are three main ingredients in our constructions.
The first ingredient is a combinatorial result that calculates the classical $\smallsmile_n$ two-variable cochain operations of Steenrod in the standard model $B\g1_2$ of the classifying space of the symmetric group $\g1_2$.
Our result implies the known calculation of Steenrod Squares in the cohomology of $B\g1_2$, without using the Cartan product formula, but it does more.
The second ingredient is the construction of certain very specific chain homotopies between pairs of chain maps from chains on $BV_4$ to chains on $B\Sigma_4$, where $V_4 \subset \Sigma_4$ is the normal subgroup of order 4 in the symmetric group.
Chain homotopies produce explicit formulae that write differences of cycles as boundaries.
The third ingredient is the exploitation of certain operadic multivariable cochain operations that extend Steenrod's two-variable $\smallsmile_n$ operations.
The operad cochain operations allow boundary formulae in the chains on $B\g1_4$ to be interpreted as coboundary relations between cocycles on any space $X$, leading to the Adem relations.
All these ingredients will be previewed in the introduction, and carefully developed in the paper.\\

The motivation for our work arose as follows.
In the work \cite{7brumfielmorgan}, on the Pontrjagin dual of 4-dimensional Spin bordism, two of the authors of this paper studied a simplicial set model of a three-stage Postnikov tower that represents this Pontrjagin dual functor.
To construct the tower as an explicit simplicial set, rather than just as a homotopy type, we made use of a degree 5 cochain $x(a)$, produced by the other author, with $dx(a) = Sq^2Sq^2(a) + Sq^3Sq^1(a)$ for a degree 2 cocycle $a$.
We wanted to understand a simplicial set delooping of that Postnikov tower related to 5-dimensional Spin bordism.
To accomplish that, we needed a degree 6 cochain $x(\alpha)$ with $dx(\alpha) = Sq^2Sq^2(\alpha) + Sq^3Sq^1(\alpha)$, for a degree 3 cocycle $\alpha$.
Despite the low degree, there does not seem to be a simple coboundary formula for even that Adem relation.
Rather than grind out by direct computer computation a coboundary formula in that one case, we embarked on our project of finding structured coboundary formulae for all Adem relations among Steenrod Squares.\\

The Spin bordism project itself, along with the 3-dimensional version in \cite{6brumfielmorgan} and the combinatorial description of $Pin^-$ structures on triangulated manifolds in \cite{8brumfielmorgan}, resulted from questions from the physicist Anton Kapustin about finding discrete cochain/cocycle level descriptions of all invariants of the Spin bordism of classifying spaces of finite groups $G$.
His questions were related to topics in condensed matter physics, connected to Spin bordism classification of principal $G$-bundles over low dimensional triangulated Spin manifolds, regarded as discrete lattice models of various phenomena \cite{11gaiottokapustin}, \cite{12kapustin}.\\

In this introductory section of the paper we give some historical background, state our main results, and roughly outline the proofs.

\subsection {Steenrod Operations}

\paragraph{0.1.1} The initial approach to cohomology operations dates back to the 1940's and 1950's.
We denote the normalized chain complexes and cochain complexes with $\FF_2$ coefficients of a simplicial set by $N_*(X)$ and $N^*(X)$.
Steenrod defined his Squares in terms of certain natural cochain operations $\smallsmile_i \colon N^p(X) \otimes N^q(X) \to N^{p+q-i}(X)$.\footnote{In this introduction, we grade cohomology in positive degrees, because that is more familiar. But in the rest of the article we follow the more appropriate convention that grades cohomology in negative degrees.} 
These operations arise by applying duality to a natural equivariant chain map $\tawd \colon N_*(E\g1_2) \otimes N_*(X) \to N_*(X) \otimes N_*(X)$ constructed by Steenrod in \cite{23steenrodproducts}.
Here, $E\g1_2$ is a specific contractible $\g1_2$-space and $\g1_2$ acts in the obvious way on the domain and switches the two factors of the range.
For a cocycle $\alpha$ of degree $k$, $Sq^{k-i}([\alpha]) = [\alpha \smallsmile_{i} \alpha]$, where $[\alpha]$ denotes the cohomology class of $\alpha$.\\

Restricted to $N_0(E\g1_2) \otimes N_*(X) = \FF_2[\g1_2] \otimes N_*(X) \to N_*(X) \otimes N_*(X)$, the map $\tawd$ is an equivariant extension of the classical Alexander-Whitney diagonal approximation.
We will describe Steenrod's map $\tawd$ in \S2.2.3 and \S2.2.4, but the theory of acyclic models easily implies equivariant extensions of any diagonal approximation do exist, and any two are equivariantly chain homotopic.
Thus all choices of equivariant maps extending some diagonal approximation lead to cochain operations like the $\smallsmile_i$ and these all define the same cohomology operations $Sq^{k-i}([\alpha])$.

\paragraph{0.1.2} Steenrod realized early on that a more general theory of cohomology operations could be formulated using the language of homology of groups, \cite{25steenrodsymmetric}.
For any group $G$ and a contractible free $G$-space $EG$, the homology of $BG$ is given by the homology of the coinvariant quotient complex $N_*(EG)_G = N_*(G \backslash EG) = N_*(BG)$, defined by setting $\tilde x \equiv g \tilde x$, all $g \in G$, $\tilde x \in N_*(EG)$.
The cohomology of $BG$ is given by the cohomology of the invariant subcomplex $N^*(EG)^G$.\\

Before sketching Steenrod's group homology approach to cohomology operations, we need a brief discussion of products.
When working with products of spaces, one needs to decide what chain complexes will be used to compute homology and cohomology.
You can use normalized chain complexes of the product spaces or you can use tensor products of complexes associated to the factors.
These choices are related by chain equivalences $AW \colon N_*(X \times Y) \to N_*(X) \otimes N_*(Y)$ of Alexander-Whitney, and $EZ\colon N_*(X) \otimes N_*(Y) \to N_*(X \times Y)$ of Eilenberg-Zilber.
Moreover, these chain equivalences are functorial in $X$ and $Y$ and associative.\\

If $X = Y$ the $EZ$ map is equivariant with respect to the interchange of factors, but the $AW$ map is not.
There are various ways to extend the $AW$ map in an equivariant manner.
For example, following Dold \cite{10dold}, one can prove using acyclic models that there exist $\g1_n$-equivariant chain homotopy equivalences
\begin{equation*}
\taw \colon N_*(E\g1_n) \otimes N_*(X \times \cdots \times X) \to N_*(X) \otimes \cdots \otimes N_*(X).
\end{equation*}
We can precompose such a map with $\Id \otimes \Delta_*$, where $\Delta \colon X \to X^n$ is the diagonal $x\, \mapsto (x, \dots, x)$.
Thus, there are equivariant cochain maps
\begin{equation*}
N^*(X)^{\otimes n} \xrightarrow{\taw^*} N^*(E\g1_n) \otimes N^*(X^n ) \xrightarrow {\Id \otimes \Delta^*} N^*(E\g1_n) \otimes N^*(X).
\end{equation*}
These maps restrict to maps of invariant subcomplexes.
If $\alpha \in N^k(X)$ is a cocycle then $\alpha^{\otimes n}$ is an invariant cocycle, which maps to a cocycle in the invariant complex $(N^*(E\g1_n) \otimes N^*(X))^{\g1_n}$, by the above composition.
The cohomology of the invariant complex $(N^*(E\Sigma_n) \otimes N^*( X))^{\Sigma_n}$ is $H^*(B\Sigma_n) \otimes H^*(X)$.\\

Thus, in the composition above, the invariant cocycle $ \alpha^{\otimes n}$ produces a cohomology class of degree $nk$ in $H^*(B\g1_n) \otimes H^*(X)$, which evaluates on homology classes in $H_i(B\g1_n) \otimes H_{nk-i}(X)$, giving a map $H_i(B\g1_n) \to H^{nk-i}(X)$.
This was Steenrod's construction, defining cohomology classes in $H^{nk-i}(X)$ as functions of $[x] \in H_i(B\g1_n)$ and $[\alpha] \in H^k(X)$.
The non-zero element $[x_i] \in H_i(B\g1_2)$ gives rise to the Steenrod Square $Sq^{k-i} ([\alpha])$.\\

A great advantage of the homology of groups approach was that it immediately led to Steenrod's reduced $p^{th}$ power operations $P^i$ of degree $2i(p-1)$, defined on the cohomology of spaces with $\FF_p$ coefficients for odd primes $p$.
These operations corresponded to certain elements in the $\FF_p$ homology of the classifying spaces of cyclic groups of order $p$ that map non-trivially to generators of the $\FF_p$ homology of the symmetric group $\g1_p$.
The specific construction mimics the discussion above, using Dold's $\taw$ map and invariant cocycles $\alpha^{\otimes p}$ in normalized cochain complexes with $\FF_p$ coefficients.\\

Then came the great work by Wu, Thom, Cartan, Serre, Adem, and others establishing all the important properties of Steenrod Squares and reduced $p^{th}$ powers as cohomology operations, such as the Cartan product formula and the Adem relations between certain sums of compositions of Steenrod operations.
It was proved that the Steenrod operations, along with cup products, generate \textit{all} natural cohomology operations defined on cohomology of spaces with $\FF_p$ coefficients.\footnote{For odd primes $p$, it is also necessary to include as a Steenrod operation the Bockstein operator $\beta$ of degree 1 associated to the coefficient sequence $0 \to \ZZ/p \to \ZZ/p^2 \to \ZZ/p \to 0$. When $p = 2$, $\beta = Sq^1$.}
The underlying cochains were pushed to the background, including Steenrod's original $\smallsmile_n$ operations in the $p = 2$ case.
The Adem relations for all $p$ were related to computations in the homology of $\g1_{p^2}$, or to computations of the cohomology of Eilenberg-Maclane spaces $K(\ZZ/p, n)$.\\

\paragraph{0.1.3} It was certainly understood in the 1950's that if one had explicit cochain formulae for the chain equivalences $AW$, $\taw$, and $EZ$, and if one exploited specific homotopies to the identity of maps $B\g1_{p^2} \to B\g1_{p^2}$ defined by conjugations, and if one had a few other explicit chain homotopies, then one could produce cochain level proofs of the Cartan formula and the Adem relations.\\

However, such chain level manipulations were not really feasible.
Not only that, but acyclic $G$-spaces $EG$ and the operations $AW$, $\taw$, $EZ$ were not even uniquely defined, but rather only up to certain kinds of homotopies, so what would cochain level proofs even mean?
The method of acyclic models implied the existence of these operations with certain homological properties, and methods emphasizing homology and cohomology, rather than chains and cochains, were in vogue.
The cohomological proofs exploited the freedom to choose different contractible spaces $EG$ with free $G$ actions, for various groups $G$.
It was understood, by the method of acyclic models, that any two contractible models for $EG$ were equivariantly homotopy equivalent, and any two equivariant maps between such models were equivariantly homotopic.
This gave the impression, forcefully stated by some authors, that cochain formulae were unnatural, and obscured the true issues.

\paragraph{0.1.4} But it turns out that there are greatly preferred choices of models of acyclic free $G$ spaces $EG$, the $AW$ and $EZ$ maps, and in the case $p = 2$ a preferred choice of the cochain operations $\smallsmile_i$, or equivalently the map $\tawd$.
In fact, the preferred choice of the $\smallsmile_i$ is Steenrod's original definition.
The second author has emphasized the important properties of Steenrod's definition of the $\smallsmile_i$, and has given axioms characterizing that choice in \cite{20medinaaxiomatic}.\\

The theory of operads extends Steenrod's $\smallsmile_i$'s to multivariable chain and cochain operations that bring the chains for symmetric groups $N_*(B\g1_n)$ into prominence, not just the homology.
Chains $x \in N_i(B\g1_n)$ determine natural cochain operations $N^k(X ) \to N^{nk - i}(X )$, linear in $x$.
Cycles $x$ take cocycles to cocycles.
If two cycles in $N_*(B\g1_n)$ differ by an explicit boundary, then the resulting cocycle operations applied to a cocycle differ by an explicit coboundary.
Precisely, $(\partial x)(\alpha) = d(x(\alpha))$.\\ 

It is our belief that use of the preferred models and preferred cochain formulae actually makes proofs of things like the Adem relations for $p=2$ easier to fully understand.
That is essentially our goal in this paper.
Our exposition in the paper is rather leisurely, and the classical constructions due originally to Steenrod will all be explained.
We will also explain the operad methods that produce cochain and cocycle operations with $\FF_2$ coefficients.

\subsection{Adem Relations}

\paragraph{0.2.1} Now we will return to the discussion of our results on Adem relations.
Of course $B\g1_2 \simeq RP^\infty$, the real projective space.
One of our main results is a cochain level computation of the $\smallsmile_n$ products in $N^*(B\g1_2)$.
First it is well-known, and easy, that in a standard model the cochain algebra $N^*(B\g1_2) = \FF_2[t]$, where $t \in N^1(B\g1_2)$ is dual to $x_1$, and $t^p \in N^p(B\g1_2)$ is dual to $x_p$.
The class $x_p$ is homologically represented by the real projective space $RP^p$.
Our first main result is this:

\begin{thm} \label{binom1}
	$t^i \smallsmile_n t^j = \binom{i}{n}\binom{j}{n} t^{i+j-n} \in N^*(B\Sigma_2).$
\end{thm}

We derive this result purely combinatorially from Steenrod's definition of the $\smallsmile_i$.
In fact, what we really evaluate is Steenrod's equivariant chain map $\tawd \colon N_*(E\g1_2) \otimes N_*(B\g1_2) \to N_*(B\g1_2) \otimes N_*(B\g1_2)$, and the cochain statement follows by duality.
Our result includes the evaluation of the Squares in projective space, $Sq^k(t^i) = \binom{i}{k}t^{i+k}$, a result that is usually deduced as a consequence of the Cartan formula for the evaluation of Squares on cup products of cohomology classes.
We do not use the Cartan formula, and we obtain more, namely the evaluation of all $\smallsmile_n$ products in a specific cochain complex $N^*(B\g1_2)$.\\

The binomial coefficients in Theorem~\ref{binom1} correspond to rather simple binomial coefficient formulae for counting various kinds of partitions of integers, either exactly or mod 2.
There is no need for strenuous binomial coefficient manipulations that one often finds in proofs of Adem relations.\footnote{An analogue of Theorem~\ref{binom1} for an equivariant map $\tawd \colon N_*(EC_p; \FF_p) \otimes N_*(BC_p; \FF_p) \to N_*(BC_p; \FF_p)^{\otimes p}$, where $C_p$ is the cyclic group of order $p$, seems to be the main obstacle to extending our results to odd primes $p$.} 

\paragraph{0.2.2} Following Adem \cite{1adem}, we study classifying spaces $BV_4 \subset BD_8 \subset B\g1_4$ at the simplicial set level, where $\g1_4$ denotes the symmetric group, $D_8$ is the dihedral group, and $V_4$ is the normal Klein 4-group $V_4 \subset \g1_4$ containing the three products of disjoint 2-cycles.
Of course $V_4 \simeq \g1_2 \times \g1_2$, generated by any two non-identity elements.
The dihedral group, of which there are three conjugate versions in $\g1_4$, contains other copies of $\g1_2 \times \g1_2$.
The $D_8$ we work with contains the commuting 2-cycles $b =(12)$ and $c = (34)$, and is generated by these elements along with $a = (13)(24) \in V_4$.
We will work with the generators $\{a, bc\} \in V_4 \simeq \g1_2 \times \g1_2$.
Note that conjugation in $\g1_4$ by the element $(23)$ interchanges the generators $a = (13)(24)$ and $bc = (12)(34)$ of $V_4.$\footnote{In the first few sections of the paper we use the disjoint cycle notation to name permutations.
Later on, it will be more natural to name a permutation $\sigma \in \Sigma_n$ as a function, written as a sequence $(\sigma(1) \sigma(2) \cdots \sigma(n))$.}\\

The homology $H_*(BV_4) \simeq H_*(B\g1_2 \times B\g1_2)$ is generated by products of projective spaces $ RP^q \times RP^p$.
In specific simplicial chain complexes, these generators are given by $x_q \times x_p = EZ(x_q \otimes x_p)$, where $EZ$ is the Eilenberg-Zilber map $$EZ\colon N_*(B\g1_2) \otimes N_*(B\g1_2) \to N_*(B\g1_2 \times B\g1_2) \simeq N_*(BV_4).$$ It is important to point out that the map $EZ$ is pretty complicated, so $x_q \times x_p$ is actually a sum of quite a large number of basic elements.
Essentially, $EZ$ amounts to triangulating prisms.\\

With the models of contractible $G$ spaces $EG$, and classifying spaces $BG$, that we use, inclusions of groups $H \subset G$ yield inclusions of simplicial sets $EH \subset EG$ and $BH \subset BG$, hence inclusions of normalized chain complexes, $N_*(EH) \subset N_*(EG)$ and $N_*(BH) \subset N_*(BG)$.
We thus have many such inclusions corresponding to our chosen subgroups $\Sigma_2 $'s $\subset V_4 \subset D_8 \subset \Sigma_4$.
If we name an element in some set associated to one of these groups, we will generally use the same name for that element viewed in the similar set associated to a larger group.\footnote{It does not seem all that controversial to call by only one name an element of a subset of various other sets, and it saves substantially on notation.} 

\paragraph{0.2.3} The cycle $x_q \times x_p \in N_{q+p}(BV_4) \subset N_{q+p}(BD_8) \subset N_{q+p}(B\Sigma_4)$ corresponds, following \S0.1.4, to a cocycle operation.
We construct an explicit $V_4$-equivariant chain homotopy $J_\Psi \colon N_*(EV_4) \to N_{*+1}(ED_8)$, which induces on coinvariants a chain homotopy $\bar{J}_\Psi \colon N_*(BV_4) \to N_{*+1}(BD_8)$, so that the cycle $x_q \times x_p + \partial \bar{J}_\Psi(x_q \times x_p) \in N_{p+q}(BD_8)$ homologous to $x_q \times x_p$ corresponds to a very specific operation expressed in terms of iterated Squares and $\smallsmile_i$ products of Squares.
The chain homotopy $J_\Psi$ is constructed using the preferred choices of $AW$ and $EZ$ maps, along with Steenrod's definition of the $\smallsmile_i$, and our explicit calculation of $\smallsmile_i$ products in $N^*(B\g1_2)$.
Although the main result concerning cocycle operations is first expressed entirely in terms of $\smallsmile_i$ products, it will be more familiar to express it in terms of Square operations on a cocycle $\alpha \in N^n(X)$.
For fixed cocycle degree $n$, the cocycle operation associated to $x_q \times x_p$ has degree $3n - (q+p)$.
Here is a precise statement.

\begin{thm} \label{Adem1}
	For a cocycle $\alpha \in N^n(X)$ we have
	\begin{equation*}
	(x_q \times x_p)( \alpha) + d(\bar{J}_\Psi (x_q \times x_p)(\alpha))\, =\
	\sum_\ell\ \binom{p-\ell}{p-2\ell} Sq^{2n-q-\ell}Sq^{n - p +\ell} (\alpha) + d (N_{q,p,n}(\alpha)),
	\end{equation*}
	where
	\begin{equation*}
	N_{q,p,n}(\alpha) \ =
	\sum_{\substack{0< a\le \ell \in \Zee[1/2];\ \\ a\equiv \ell\ mod\ \Zee}}\binom{p-\ell-a}{p-2\ell} \binom{p-\ell +a}{p-2\ell}Sq^{n -p+\ell +a}(\alpha)\ \smallsmile_{q-p+2\ell+1}\ Sq^{n -p + \ell -a}(\alpha).
	\end{equation*}
	By symmetry, we also have an explicit formula
	\begin{equation*}
	(x_p \times x_q)(\alpha) + d (\bar{J}_\Psi(x_p \times x_q)(\alpha))\, =\
	\sum_\ell\ \binom{q-\ell}{q-2\ell} Sq^{2n-p-\ell}Sq^{n - q +\ell} (\alpha) \ + \
	d (N_{p,q,n}(\alpha)).
	\end{equation*}
\end{thm}

In these formulae, $d$ is the cochain coboundary in $N^*(X)$.
The existence of the homotopies $J_\Psi$ and $\bar{J}_\Psi$ is a special case of a general result, Theorem~\ref{t:homotopy for classifying spaces}, proved in \S1.3.1.
A general formula for these homotopies is given in Theorem~\ref{t:homotopy for iota and Psi}.
Study of the underlying $V_4$-equivariant map in our special case, $\Psi \colon N_*(EV_4) \to N_*(ED_8)$, is carried out in the subsections of Section 3.
This is where our explicit calculation of Steenrod's map $\tawd \colon N_*(E\g1_2) \otimes N_*(B\g1_2) \to N_*(B\g1_2) \otimes N_*(B\g1_2)$ gets used.

\paragraph{0.2.4} The purely cohomological part of Theorem~\ref{Adem1}, that is, the terms involving compositions of Squares, was established by Adem.
But the proof was not easy, and was rather hidden by his use of chain homotopies that were not made explicit.
Our proof, even at the cochain level with coboundary terms, is quite transparent once the $\smallsmile_n$ products in $B\g1_2$ are computed.
The binomial coefficients in Theorem~\ref{Adem1} are essentially just repetitions of the binomial coefficients in Theorem~\ref{binom1} with different values of $i,j,n$.

\paragraph{0.2.5} We have observed in \S0.2.2 that conjugation by $(23) \in \g1_4$ interchanges our two chosen $\g1_2$'s in $V_4 \simeq \g1_2 \times \g1_2$.
But an inner automorphism $c_g$ of a group $G$ gives rise to a very specific equivariant homotopy $J_g$ on $EG$ that induces a homotopy $\bar{J}_g$ between $\Id \colon BG \to BG$ and $Bc_g \colon BG \to BG$.
In our case, we then have a formula at the cycle level$$x_q \times x_p + x_p \times x_q = \partial \bar{J}_{(23)} (x_q \times x_p) \in N_*(B\Sigma_4).$$
For $g\in G$, we give the general formula for $ J_g$ in \S1.3.3.\\

We apply this last equation to a cocycle $\alpha$ and add that equation to the two equations in Theorem~\ref{Adem1}.
The result is the following cochain level formula:

\begin{thm} \label{Adem2}
	For a cocycle $\alpha \in N^n(X)$ we have
	\begin{multline*}
	\sum_\ell \binom{q-\ell}{q-2\ell} Sq^{2n-p-\ell} Sq^{n-q+\ell} (\alpha) \ + \ 
	\sum_\ell \binom{p-\ell}{p-2\ell} Sq^{2n-q-\ell} Sq^{n-p+\ell} (\alpha) \\ = \
	d \Big( (\bar{J}_\Psi (x_p \times x_q)(\alpha) +
	N_{p, q,n}(\alpha) +
	\bar{J}_\Psi (x_q \times x_p)(\alpha) +
	N_{q, p, n}(\alpha) +
	\bar{J}_{(23)} (x_q \times x_p)(\alpha) \Big).
	\end{multline*}
\end{thm}

And there you have it, a very specific cocycle/coboundary formulation of Adem relations.\\

One interesting thing about the equivariant chain homotopies $J_\Psi$ and $J_{(23)}$ is that they are essentially given by exactly the same kind of formula, arising from a rather general common underlying situation.
Specifically, we mentioned above that any two equivariant maps between contractible free $G$ complexes are equivariantly homotopic.
But with preferred models there is even a canonical way to choose equivariant homotopies.
This will all be explained in due course.\\

Of course to fully explain Theorems~\ref{Adem1} and~\ref{Adem2}, we need to clarify exactly how elements of $N_*(B\g1_4)$ act on cocycles.
This will be accomplished using operad methods, which we begin discussing in Subsection~\ref{Operads}.

\paragraph{0.2.6} We want to emphasize that the overall structure of our proof of the Adem relations follows very closely the original proof of Adem \cite{1adem}.
As mentioned previously, Adem proved the cohomological part of Theorem~\ref{Adem1}.
Our combinatorial result Theorem~\ref{binom1} provides cochain details Adem lacked.
In addition, in his proof of the cohomological part of Theorem~\ref{Adem1}, Adem did not need an explicit $V_4$-equivariant chain homotopy like our $J_\Psi$, only that such a thing existed.
Finally, Adem also used the inner automorphism by the element $(23) \in \g1_4$ to show that the elements $x_q \times x_p$ and $x_p \times x_q$ that are distinct in the homology of $V_4$ map to the same element in the homology of $\g1_4$, hence differ by a coboundary and define the same cohomology operation.
He did not have the structured method using operads to generate specific coboundary formulae, but just the fact that a cohomology operation is determined by a homology class in the symmetric group surely corresponded  to some theoretically possible behind the scenes explicit coboundary computations.

\paragraph{0.2.7} The relations in Theorem~\ref{Adem2} do not look like the usual Adem relations that express inadmissible compositions $Sq^aSq^b$ of Steenrod Squares, with $ a < 2b$, as sums of admissible compositions.
The easiest way to recover the usual Adem relations is to use the fact that the Steenrod Squares are stable cohomology operations, that is, they commute with cohomology suspension.
One looks at a high suspension $s^N[\alpha] \in H^{n+N}(S^N \wedge X)$ and identifies one of the relations for degree $n+N$ cocycles in Theorem~\ref{Adem2} so that the binomial coefficients give exactly the usual Adem relation for $Sq^aSq^b(s^N[\alpha)].$\footnote{We found this trick in some course notes of J. Lurie \cite{13lurie}, although probably goes back to Adem.} We include this argument in an appendix.

\paragraph{0.2.8} Theorem~\ref{Adem2} produces, for each cocycle degree $n$ and pair $(p, q)$, a relation between Steenrod operations of total degree $3n - (p+q)$.
With $n$ and the sum $p+q$ fixed, these relations are highly redundant.
We do not see how to use linear combinations of these relations to single out a preferred coboundary formula for a given Adem relation on degree $n$ cocycles.
The desuspension trick of \S0.2.7 also does not single out preferred coboundary formulae for Adem relations.
Of course two different coboundary expressions for a fixed relation differ by a cocycle, related to the indeterminacy of secondary cohomology operations, and that might be something interesting to pursue.

\subsection{Operads} \label{Operads}

\paragraph{0.3.1} In the early 2000's the Surj operad was introduced, along with actions on tensors of cochains generalizing Steenrod's $\smallsmile_n$ operations.
The Surj operations are multivariable cochain operations, computed as sums of products of evaluations of different cochains on faces of simplices (coface operations).
The sums are parameterized by certain diagrams.
We review the sum over diagrams formulae for the $\smallsmile_n$ in Subsection 2.2 and for the more general Surj operations in \S4.3.2.\\

The results of this paper lead to computer algorithms for producing explicit Surj operad formulae that can be used to define cochains whose boundaries are Adem relations, in the form stated in Theorem~\ref{Adem2}.
The formulae quickly involve sums of a very large number of terms as the degrees of the relations and the cocycles increase.
 We do believe that it is a theoretical advance to have the better understanding of Adem relations that we establish in this paper.
But it has limited practical use.

\paragraph{0.3.2} We will briefly review how operad methods give rise to an action of $N_*(E\g1_n)$ as multivariable cochain operations.
We take this up in greater detail in Section~\ref{sect4}.
Given a space $X$, we make use of three operads and operad morphisms\footnote{Strictly speaking we only use the $\FF_2$ coefficient versions of these operads.}
$$\cE \xrightarrow{TR} \cS \xrightarrow{Eval} End(N^*(X)).$$
The first is the Barratt-Eccles operad, with $\cE_n = N_*(E\g1_n)$, the normalized chains on the classical MacLane model of a contractible $\g1_n$ simplicial set.
The second operad is the Surj operad, also called the Step operad or the Sequence operad, whose construction and action on tensors of cochains directly generalizes Steenrod's original construction of the $\smallsmile_i$ products.
The third operad $End(C^*)$ for a cochain complex $C^*$ has operad components $End(C^*)_n = \ \mathrm{Hom}((C^*)^{\otimes n}, C^*)$, the multilinear cochain operations.
We remark that our constructions are applicable to any Barratt-Eccles algebra, i.e., a chain complex $A$ together with an operad morphism $\cE \to End(A)$.
The Adem cochains described in this paper are in this general case only as constructive as the morphism defining the algebra structure.

\paragraph{0.3.3} The operad $\cE$ was introduced in \cite{2barratteccles}, and studied in great detail by Berger and Fresse \cite{3bergerfresse}, \cite{4bergerfresse}.
The Surj operad $\cS$ was more or less simultaneously introduced by Berger and Fresse and by McClure and Smith \cite{18mccluresmith}.
The operad components $\cS_n$ are free acyclic $\g1_n$ chain complexes with $\FF_2$ basis in degree $k$ named by surjections $s \colon\{1, 2, \ldots, n+k\} \to \{1, 2, \ldots, n\}$ with $s(i) \not= s(i+1)$, all $i$.
Each such surjection gives rise to a multivariable cochain operation $N^*(X)^{\otimes n} \to N^*(X)$ that lowers total degree by $k$, (when cochains are graded positively), and these operations define the operad morphism $Eval$.
When $n = 2$, the surjection named $(1212 \ldots)$ with $2+k$ entries corresponds to Steenrod's operation $\smallsmile_k$.
The operad morphism $TR$ is called table reduction and was introduced by Berger and Fresse as a way to explain how simplices in $E\g1_n$, which can be regarded as tables consisting of some number of permutations, act on tensors of cochains, via the Surj operad.
The important thing for us is that our chain homotopies $J_\Psi \colon N_*(EV_4) \to N_{*+1}(E\g1_4)$ and $J_{(23)} \colon N_*(E\g1_4) \to N_{*+1}(E\g1_4)$ take values in $\cE$.
But by $TR$ they are pushed over to $\cS$, and then to cochain operations.
We cannot see the needed chain homotopies directly in $\cS$.

\paragraph{0.3.4} In the paper \cite{19medinacartan}, the second author carried out a program, somewhat similar to the joint program here, for finding an explicit coboundary formula implying the Cartan product formula for Steenrod Squares.
In fact, both that work and the work in this paper originated when we were working on the paper \cite{7brumfielmorgan} and needed a specific coboundary formula for the relation $Sq^2([a]^2) = (Sq^1[a])^2$ for a cocycle $a$ of degree 2, which is simultaneously a Cartan relation and an Adem relation.
Since that time, we have figured out our much more structured explanations of both the Cartan formula and the Adem relations in general for the mod 2 Steenrod algebra.

\paragraph{0.3.5}

We did apply our computer algorithm to Theorem~\ref{Adem2} with $(q, p) = (4,1)$ to find a formula $dx(\alpha) = Sq^2Sq^2(\alpha) + Sq^3Sq^1(\alpha)$ for a degree 3 cocycle $\alpha$.
The class $x(\alpha)$ desuspends to the class we used in \cite{7brumfielmorgan} for the corresponding Adem relation in one lower dimension.\\

In general, Theorem~\ref{Adem2} with $(q, p) = (4,1)$ ultimately produces, for each cocycle $\alpha$ of degree $n$, a formula writing a certain sum of Steenrod operations of total degree $3n - 5$ applied to $\alpha$ as the coboundary of a cochain $x(\alpha)$ of degree $4n-6$.
The Surj formula used to define the cochain  $x(\alpha)$ is a sum of 116 surjections $\{1, 2, \ldots, 10\} \to \{1,2,3,4\}$ applied to the multi-tensor $\alpha^{\otimes 4}$, via the operad morphism $Eval$.
Each such surjection term gives a sum of coface operations parameterized by certain diagrams.
The number of diagrams, but not the surjections themselves, depends on $n$.\\

When $n = 3$, the diagrams for the Surj formula for the degree 6 cochain $x(\alpha)$ all yield 0, except for 26 of the 116 surjections $\{1, 2, \ldots, 10\} \to \{1,2,3,4\}$.
The number of diagrams, or coface expressions, associated to these 26 surjections is around 100.
So the coface operation formula for $x(\alpha)$ with $dx(\alpha) = Sq^2Sq^2(\alpha) + Sq^3Sq^1(\alpha)$ is quite complicated, even for cocycles $\alpha$ of degree 3.

\paragraph{0.3.6} The paper \cite{9chataurlivernet} also makes use of operad methods to study actions of the mod 2 Steenrod algebra.
Those authors' goal was not to produce explicit coboundary formulae for the Cartan and Adem relations, but rather to develop a universal operadic treatment of contexts in which the Steenrod algebra acts, or equivalently, in which the Cartan and Adem relations hold for some family of operations, somewhat similar to what May carried out in his paper \cite{17maygeneeral}.\\
  
We believe our work is related to results that connect operads and higher structures on cochain algebras to homotopy theory, as in \cite{14mandellpadic}, \cite{15mandellinteger}, \cite{16mayoperads}, \cite{22smith}.
It is clear that our proofs of Theorems \ref{Adem1} and \ref{Adem2}, especially the construction of the equivariant chain map $\Psi\colon N_*(EV_4) \to N_*(E\g1_4)$ and the equivariant chain homotopy $J_\Psi\colon N_*(EV_4) \to N_{*+1}(E\g1_4)$ between $\Psi$ and the inclusion map, bring into the open certain useful structure inside the Barratt-Eccles operad that has perhaps not previously been noticed.

\section*{Acknowledgment}

We thank the referee for their careful reading of this work and for many valuable suggestions improving its exposition.

\section{Simplicial Sets and Classifying Spaces}
In this section we review some basic facts about simplicial sets, their chain and cochain complexes, and some facts about classifying spaces $BG$ for groups arising from acyclic spaces $EG$ with free $G$ actions.
We also explain how using the join operation, the Eilenberg-Zilber and Alexander-Whitney maps leads to explicit chain homotopies between maps whose target is the chains on $EG$.

\subsection{Simplicial Sets and Normalized Chain Complexes}

\paragraph{1.1.1} Recall that a simplicial set $X$ consists of a collection of sets $\{X_n\}_{n \geq 0}$ indexed by the natural numbers, together with face and degeneracy operators.
If $X$ is a simplicial set and $\sigma$ is in $X_n$, the set of ``$n$-simplices" of $X$, then there is a unique simplicial map $s \colon \Delta^n \to X$ sending the top-dimensional simplex to $\sigma$, where $\Delta^n$ is the $n$-simplex, regarded as a simplicial set with standard face and degeneracy operators.
In particular, the vertices of $\Delta^n$ are $\{0 < 1 < \cdots < n\}$ and the elements of $\Delta^n_k$ are non-decreasing sequences $(i_0, i_1, \ldots, i_k)$ of vertices.
The degenerate simplices are those that repeat at least one vertex.
The degenerate $n$-simplices $\sigma \in X_n$ are images of degenerate $n$-simplices of $\Delta^n$ under the associated maps $s \colon \Delta^n \to X$.
Products of simplicial sets are simplicial sets with $(X \times Y)_n = X_n \times Y_n$, with the product face and degeneracy operators.
We sometimes use the word space for a simplicial set.
Maps of spaces always means maps of simplicial sets.

\paragraph{1.1.2} Naturally associated to simplicial sets are various chain and cochain complexes, which have homology and cohomology groups.
For any simplicial set $X$, let $N_*(X)$ denote the chain complex of normalized chains with $\FF_2$ coefficients.
Thus in degree $n$, $N_n(X)$ is the $\FF_2$ vector space with basis the set of $n$-simplices $X_n$, modulo the subspace generated by the degenerate simplices.
 
\paragraph{1.1.3} The boundary operator $\partial$ on $N_*(X)$ is defined on a basic $n$-simplex to be the sum of the codimension one faces.
Given $\sigma \in X_n$ there is the unique simplicial map $s\colon \Delta^n \to X$ with $s(\Delta^n) = \sigma$.
We can thus define the boundary formula universally by $$\partial(i_0,i_1, \ldots, i_n) = \sum_{j=0}^n\ \partial_j(i_0, i_1, \ldots, i_n) = \sum_{j=0}^n\ (i_0, \ldots \widehat{i_j} \ldots, i_n),$$ where $\widehat{i_j} $ means $i_j$ is deleted.
It is easy to check that if $\sigma$ is a degenerate simplex then $\partial(\sigma) = 0$.

\paragraph{1.1.4} The normalized cochains are defined as the dual chain complex $N^*(X) = Hom(N_*(X), \FF_2)$, where we regard $\FF_2$ as a chain complex concentrated in degree 0.
Cochains thus lower degree, so the natural grading then is to regard $N^*(X)$ as a chain complex concentrated in negative degrees, and this will be our convention.\footnote{Negative means $\leq 0$ and strictly negative is $< 0$.
The negative grading of cochains, which is the correct way to do it, does require some extra thought at times by those accustomed to positive grading.} Since cochains are in negative degrees, the adjoint coboundary operator $d$ on $N^*(X)$, also lowers degree by one.
So everything is a chain complex, and cochain complex just means a chain complex concentrated in negative degrees.

\paragraph{1.1.5} The Eilenberg-Zilber functorial chain map $EZ \colon N_*(X) \otimes N_*(Y) \to N_*(X \times Y)$ with $\FF_2$ coefficients will play a prominent role in the paper.
By naturality, the general definition follows from the case $X= \Delta^n,\ Y = \Delta^m$.
In this case, we simply triangulate the prism $\Delta^n \times \Delta^m$ as the simplicial complex underlying the product of posets $(0 < 1< \cdots < n) \times (0 < 1< \cdots < m)$.
The map $EZ$ takes the tensor product of the universal simplices of dimensions $n$ and $m$ to the sum of the maximal dimension simplices of the product space.
These $n+m$ dimensional simplices correspond to strictly increasing vertex sequences in the product poset order, $(00 = i_0j_0 < i_1j_1 < \cdots < i_{n+m}j_{n+m} = nm)$, where at each step one of the indices increases by 1 and the other is unchanged.
One can prove that $EZ$ is a chain map either by working directly with these simplices or by thinking geometrically and using $\partial (\Delta^n \times \Delta^m) = \partial(\Delta^n) \times \Delta^m \cup \Delta^n \times \partial(\Delta^m)$.
We are working with $\FF_2$ coefficients throughout, so orientation signs are irrelevant.\\

In the case $m = 1$, denote by $EZ_h$ the composition $EZ_h(z) = EZ(z \otimes (0,1))$ displayed below,
\begin{equation*}
N_*(Z) \xrightarrow{\Id \otimes (0,1)} N_*(Z ) \otimes N_1(\Delta^1) \xrightarrow{EZ} N_{*+1}(Z \times \Delta^1).
\end{equation*}
Then $EZ_h$ is a (classical) universal chain homotopy between top and bottom faces of cylinders,
\begin{equation*}
(EZ_h \circ \partial + \partial \circ EZ_h) (z) = (z, (1, \ldots, 1)) + (z, (0, \ldots, 0)) \in N_*(Z \times \Delta^1).
\end{equation*}
By naturality, it suffices to verify this when $Z = \Delta^n$ and $z = (0, 1, \ldots, n)$.
Since $EZ$ is a chain map, we have
\begin{equation*}
\partial \circ EZ(z \otimes (0,1)) = EZ(\partial z \otimes (0,1)+ z \otimes 1 + z \otimes 0))
\end{equation*}
and the result is clear.\\

On the other hand, to make this more precise, one can write out sums for $EZ_h(z)$ and  $\partial \circ EZ_h(z)$ in the chain homotopy formula.
The $n+1$ dimensional simplices of the convex prism $\Delta^n \times \Delta^1$ are joins of simplices on the bottom and top faces,
\begin{equation*}
(00,10, \ldots, j0) * (j1, \ldots, n1) = (00, \ldots, j0, j1, \ldots, n1).
\end{equation*}
Then $EZ_h(z ) = \sum_j (00,10, \ldots, j0) * (j1, \ldots, n1)$ and $\partial EZ_h(z) = \sum_j \partial \big( (00,10, \ldots, j0) * (j1, \ldots, n1) \big)$.
Since the join operation on simplices just writes the vertices of one simplex after those of another it is easy to understand the boundary of the join of two simplices,
\begin{equation*}
\partial (a * b) = \partial a * b + a * \partial b + \epsilon(a) \, b + a \, \epsilon(b),
\end{equation*}
where on the right the join operation is extended bilinearly and where $\epsilon \colon N_*(\Delta^n)\to \FF_2$ is the linear map defined by $\epsilon(v) =1$ for vertices and is $0$ for positive degree simplices.\\

The join operation on simplices extends  to a linear map of degree one $* \colon N_*(\Delta^n) \otimes N_*(\Delta^n) \to N_*(\Delta^n)$, and the boundary formula above holds for all $a \otimes b$.
It follows immediately that
\begin{equation*}
(\partial * + * \partial )(a \otimes b) = \epsilon(a) b + a \epsilon (b).
\end{equation*}
The functorial Alexander-Whitney diagonal chain map $\awd \colon N_*(Z) \to N_*(Z) \otimes N_*(Z)$ will also play a prominent role.
Again by naturality it suffices to consider $Z = \Delta^n$ and $z = (0, \ldots, n)$.
The formula is
\begin{equation*}
\awd(0, \ldots n) = \sum_j \ (0, \ldots j) \otimes (j, \ldots n).
\end{equation*}
For arbitrary simplicial sets $Z$ this gives the well-known sum of front faces of simplices tensored with back faces.
A direct computation shows that $\awd$ is a chain map.

\paragraph{1.1.6} Consider a connected simplicial set $X$ for which a degree one join operation $* \colon N_*(X) \otimes N_*(X) \to N_*(X)$ has meaning, and satisfies the boundary formula $\partial (a \ast b) = \partial a \ast b + a \ast \partial b + \epsilon(a)\, b + a\, \epsilon(b)$.
Suppose we have two degree zero chain complex morphisms $\phi_0, \phi_1 \colon N_*(Z) \to N_*(X)$ for some simplicial set $Z$.
Given a simplex generator $z \in Z_n$, let $^j \partial z$ mean front faces and let $\partial^{n-j}z$ mean back faces, of degrees $j$ and $n-j$.
Consider
\begin{equation*}
J_\Phi (z) = \sum_j \ \phi_0(^j\partial z) * \phi_1 (\partial^{n-j} z) \in N_{*+1}(X).
\end{equation*}

\begin{lem} \label{l:join for homotopies}
	If $Z$ is connected and if both $\phi_0$ and $\phi_1$ induce non-zero maps in degree 0 homology, then
	\begin{equation*}
	(J_\Phi \circ \partial + \partial \circ J_\Phi) (z) = \phi_1(z) + \phi_0(z).
	\end{equation*}
\end{lem}

\begin{proof}
	We will give two proofs, one exploiting $\awd$ and the other $EZ_h$.
	The assumption that both $\phi_0$ and $\phi_1$ induce non-zero maps in degree 0 homology implies that for any vertex $v$ of $Z$ both $\phi_0(v)$ and $\phi_1(v)$ are sums of an odd number of vertices of $X$, hence $\epsilon \phi_0(v) = \epsilon \phi_1 (v) = 1$.\\
	
	For the first proof we notice that the degree one map $J_\Phi$ is the composition
	\begin{equation*}
	N_*(Z) \xrightarrow{\awd} N_*(Z) \otimes N_*(Z) \xrightarrow {\phi_0 \otimes \phi_1} N_*(X) \otimes N_*(X) \xrightarrow{*} N_*(X).
	\end{equation*}
	
	Since the first two maps are chain maps, the boundary formula for $*$ implies
	\begin{align*}
	(J_\Phi \circ \partial + \partial \circ J_\Phi) (z) & = 
	(\partial * + * \partial)\circ \phi_0 \otimes \phi_1 \circ \awd(z) \\ & =
	\phi_1(z) + \phi_0(z),
	\end{align*}
	where, denoting $z_0 = \, ^0\partial z$ and $z_n = \partial^0 z$, we used for the last equality the fact that $\epsilon \phi_0(z_0) = \epsilon \phi_1(z_n) = 1$.\\
		
	For the second proof we first define a degree zero chain map $$H_\Phi \colon N_*(Z \times \Delta^1) \to N_*(X)$$ that agrees for any chain maps $\phi_0$ and $\phi_1$ with $(\epsilon \phi_0(z_0)) \phi_1(z)$ on the top copy $Z \times 1$ of the cylinder and with $\phi_0(z) (\epsilon \phi_1(z_n))$ on the bottom copy $Z \times 0$.
	Consider an $n$-simplex in $Z \times \Delta^1$, not on the top or bottom, $$(t,u) = (t_0, \ldots, t_\ell, t_{\ell+1}, \ldots, t_n), (0 \ldots, 0, 1, \ldots, 1),$$ where there are $\ell+1$ zeros and $n-\ell$ ones in the $\Delta^1$ factor $u$.
	Set
	\begin{equation*}
	H_\Phi(t,u) = \phi_0 (t_0, \ldots, t_\ell) * \phi_1(t_{\ell+1}, \ldots, t_n) \in N_n(X).
	\end{equation*}
	Using the join boundary formula, it is an exercise to prove that $H_\Phi$ is indeed a chain map.\\
		
	We then observe that the degree one map $J_\Phi$ is the composition
	\begin{equation*}
	N_n(Z) \xrightarrow{EZ_h} N_{n+1}(Z \times \Delta^1) \xrightarrow{H_\Phi} N_{n+1}(X).
	\end{equation*}
	Since $H_\Phi$ is a chain map, with the assumption that $\epsilon \phi_0(z_0) = \epsilon \phi_1(z_n) = 1$ we have again
	\begin{align*}
	(J_\Phi \circ \partial + \partial \circ J_\Phi) (z) & =
	H_\Phi \circ ( EZ_h \circ \partial + \partial \circ EZ_h ) (z) \\ & =
	H_\Phi (z, (1, \ldots, 1)) + (z, (0, \ldots, 0)) \\ & =
	\phi_1(z) + \phi_0(z)
	\end{align*}
	as claimed.
\end{proof}

The existence of a join operation on $N_*(X)$ implies $H_*(X; \FF_2) \simeq H_*(point; \FF_2)$.
We remark that Lemma \ref{l:join for homotopies} can be proven with this much weaker hypothesis, using a different chain homotopy $J_\Phi$.
But the chain homotopy obtained using a join operation leads to the construction of equivariant chain homotopies in the presence of group actions, which will be very important in later sections of the paper.
Another reason the join is important is, as was shown in \cite{medina2020prop1}, that the join together with the Alexander-Whitney diagonal and the augmentation $\epsilon$ define an $E_\infty$ structure on chain complexes of simplicial sets, generalizing the one we describe in Section 4.3 below.

\subsection{Classifying Spaces for Discrete Groups}

Let $G$ be a discrete group.
We review here the MacLane model for a contractible left $G$ simplicial set $EG$ and a classifying simplicial set $BG = G \backslash EG$.

\paragraph{1.2.1} The MacLane model for $EG$ has $(EG)_n$ equal to the set of ordered $(n+1)$-tuples $(g_0, \ldots, g_n)$ of elements of $G$.
The $i^{th}$-face map is given by deleting the entry $g_i$ and the $i^{th}$ degeneracy map is given by repeating $g_i$.
In particular, the degenerate simplices of $EG$ are those sequences that have the same group element as two successive entries.\\

There is a left action $G\times EG\to EG$ given by
\begin{equation*}
g(g_0, \ldots ,g_n)=(gg_0, \ldots, gg_n).
\end{equation*}
This is a free action and the quotient is $BG$, the MacLane model for the \textit{classifying space} for $G$.
We can identify $BG_n$ with $n$-tuples of elements in $G$ by identifying the orbit of $(g_0, \ldots ,g_n)$ with $[g_0^{-1}g_1, \ldots ,g_{n-1}^{-1}g_n]$.
With this representation of elements in $BG_n$ the face maps are given by
\begin{align*}
\partial_0 [h_1, \ldots, h_n] & = [h_2, \ldots, h_n], \\
\partial_i [h_1, \ldots, h_n] & = [h_1, \ldots, h_{i-1}, h_ih_{i+1}, \ldots,h_n], \\
\partial_n [h_1, \ldots, h_n] & = [h_1, \ldots, h_{n-1}]. 
\end{align*}
The degenerate simplices in $BG$ are those sequences in which at least one entry is the identity element.

\paragraph{1.2.2} This definition of $EG$ is a special case of the classifying space of a category, the category which has the elements of group $G$ as objects and for every pair $g,g'\in G$ a unique morphism between them.
A simple proof that $EG$ is contractible is given by the observation that every element $g \in G$ is a terminal object of the category underlying $EG$.\\ 

$BG$ is the classifying space of the quotient category of $EG$, the quotient having a single object $*$ with ${\rm Hom}(*,*) = G$ and with the composition operation being the product in $G$.
One can think of the map $EG \to BG$ as the map of classifying spaces associated to the functor between the underlying categories that takes the morphism $(g_0, g_1)$ to the morphism $g_0^{-1}g_1$.
This assignment does preserve compositions since $g_0^{-1}g_2 = (g_0^{-1}g_1)(g_1^{-1}g_2)$.

\subsection{Equivariant Maps at the $EG$ Level}

\paragraph{1.3.1} We will next prove a rather general fact stating that certain pairs of equivariant chain maps are equivariantly chain homotopic.
Actually, results like this are well-known, using acyclic model methods.
But we want explicit equivariant chain homotopies.
Our result will be crucial for the ultimate goal of producing coboundary formulae for Adem relations.

\begin{thm} \label{t:homotopy for classifying spaces}
	Consider a group homomorphism $\iota \colon H \to G$ between finite groups, inducing a simplicial map $\iota \colon EH \to EG$, and a chain map $\iota_* \colon N_*(EH) \to N_*(EG)$.
	Note that $\iota_*$ is $\iota$-equivariant for the free actions of $H$ on the domain and of $\iota H \subset G$ on the range.
	Suppose $\Psi \colon N_*(EH) \to N_*(EG)$ is any $\iota$-equivariant chain map that induces the identity on
	\begin{equation*}
	H_0(EH) = \FF_2 = H_0(EG).
	\end{equation*}
	Then $\iota_*$ and $\Psi$ are chain homotopic, by an $\iota$-equivariant chain homotopy $J_\Psi \colon N_*(EH) \to N_{*+1}(EG)$.
\end{thm}

\begin{proof}
	First we indicate a reason this should be true somewhat different from the usual acyclic model argument.
	The equivariant chain map $\Psi$ can be regarded as a zero-cycle in the chain complex
	$$\mathrm{Hom}_{\FF_2[H]}(N_*(EH), N_*(EG)).$$
	Since $N_*(EH)$ is free over the group ring $\FF_2[H]$ and $EG$ is contractible,
	$$H_0(\mathrm{Hom}_{\FF_2[H]}(N_*(EH), N_*(EG))) = 
	H_0(\mathrm{Hom}(N_*(BH), \FF_2)) = \FF_2.$$
	Thus, $ \mathrm{Hom}(N_*(EH), N_*(EG))$ contains only one non-trivial equivariant homotopy class.
\end{proof}

\paragraph{1.3.2} The explicit equivariant chain homotopy between $\Psi$ and $\iota_*$ that we write down in the next theorem is a special case of the chain homotopies constructed in Lemma~\ref{l:join for homotopies} of \S1.1.6.
Note the join map of simplices extends to a multilinear map $N_j(EG) \otimes N_k(EG) \xrightarrow{*} N_{j+k+1}(EG)$ satisfying $\partial(x*y) = \partial x *y + x * \partial y + \epsilon(x) \, y + x \, \epsilon(y)$.
 
\begin{thm} \label{t:homotopy for iota and Psi}
	With $\iota$ and $\Psi$ as in Theorem~\ref{t:homotopy for classifying spaces}, a canonical equivariant chain homotopy between $\iota_*$ and $\Psi$ is given by the formula
	\begin{equation*}
	J_\Psi(h_0, h_1, \ldots, h_n) = \sum_j \iota_* ( h_0, h_1, \ldots, h_j) * \Psi(h_j, h_{j+1}, \ldots, h_n).
	\end{equation*}
\end{thm}

\begin{proof}
	The $\iota$ equivariance of $J_\Psi$ is obvious from the formula defining it.
	That $J_\Psi$ defines a chain homotopy is a direct consequence of Lemma~\ref{l:join for homotopies} after noticing that $^j \partial (h_0, \dots, h_n) = (h_0, \dots, h_j)$ and $\partial^j (h_0, \dots, h_n) = (h_j, \dots, h_n)$.
	\footnote{We discussed aspects of this proof of Theorem~\ref{t:homotopy for iota and Psi}, perhaps prematurely, in \S1.1.6, because we wanted to pave the way for Theorems~\ref{t:homotopy for classifying spaces} and \ref{t:homotopy for iota and Psi}. One can also prove Theorem~\ref{t:homotopy for iota and Psi} by a lengthy but straightforward direct computation.}
\end{proof}

\paragraph{1.3.3} A special case of Theorem~\ref{t:homotopy for classifying spaces} is the map on chains induced by the right translation map of simplicial sets $EG \to EG$ given by $x \mapsto xg^{-1}$, with $\iota = \Id \colon G \to G$.
In this case, one actually obtains an equivariant homotopy $J_g \colon EG \times \Delta^1 \to EG$ between the identity and the map $(g_0, \ldots, g_n) \mapsto (g_0g^{-1}, \ldots, g_ng^{-1}) = g^{-1}(c_g(g_0), \ldots, c_g(g_n))$, where $c_g$ is conjugation by $g$.
Applying equivariance, these constructions project to a homotopy $\bar{J}_g\colon BG \times \Delta^1 \to BG$ between the identity and the map induced by inner automorphism $c_g$.
The chain homotopy $N_*(EG) \to N_{*+1}(EG)$ produced by Theorem~\ref{t:homotopy for iota and Psi} coincides with the chain homotopy produced by the space level homotopy.
The formula is
\begin{equation*}
J_g(g_0, \ldots, g_{n}) = \sum_j (g_0, \ldots, g_{j}, g_{j}g^{-1}, \ldots, g_{n}g^{-1}).
\end{equation*}
 
\subsection{The Spaces $E\g1_2$ and $B\g1_2$}

\paragraph{1.4.1} Of particular importance for us is the group $G = \g1_2 = \{1, T\}$.
In $E\g1_2$ there are just two non-degenerate simplices in each dimension, which are $\tilde{x}_p = (T^0, T^1, \ldots, T^p) = (1, T, 1, \ldots, T^p)$ and $T\tx_p = (T, 1, T, \ldots, TT^p)$.
The (equivariant) boundary in the normalized chain complex is determined by $\partial \tx_p = T\tx_{p-1} + \tx_{p-1}$, since all codimension one faces except the first and last are degenerate.
In the classifying space $B\g1_2$ there is a single non-degenerate simplex in each dimension, namely $x_p = [ T, T, \ldots , T]$, which is a cycle.
For normalized cochains, we have the dual basis elements $t_p = x_p^* \in N^{-p}(B\g1_2)$ and $\tilde{t}_p = \tx_p^*,\ T\tilde{t}_p = T\tx_p^* \in N^{-p}(E\g1_2)$.

\paragraph{1.4.2} \textbf{REMARK.} We make an observation about the simplices of $E\g1_2$ and $B\g1_2$ that will be quite important later.
This observation extends the above observation about the chain boundary formula.
In $E\g1_2$, if we delete an interior interval, consisting of an odd number of adjacent vertices of a non-degenerate simplex, the result is a degenerate simplex.
If we delete an interior interval consisting of an even number of vertices, or \textit{any} initial or terminal interval of vertices, the result remains a non-degenerate simplex.
\\

It is slightly trickier understanding compositions of face operations in $B\g1_2$.
Of course one can always just apply the observations in the paragraph above about degenerate and non-degenerate simplices in $E\Sigma_2$ to the projection $E\g1_2 \to B\g1_2$.
But one can also reason directly in $B\g1_2$.
From the general face operator formulae for $BG$ in \S1.2.1, the first or last basic face operator $\partial_0$ or $\partial_p$ in $B\g1_2$ applied to $x_p = [T, T, \ldots, T]$ just deletes a first or last $T$, leaving the non-degenerate $x_{p-1}$.
But an interior basic face operator multiplies two adjacent $T$'s, resulting in a $1$ entry and a degenerate simplex.
If now another adjacent interior face operator is applied, the 1 and a $T$ are multiplied, eliminating the 1 and resulting in the non-degenerate $x_{p-2}$.
Thus a composition of adjacent interior face operators applied to $x_p$ results in a non-degenerate simplex if and only if the number of adjacent interior face operators composed is even.

\section{Cup$_n$ Products}
In the first two subsections of this section we explain Steenrod's explicit enhanced Alexander-Whitney chain map $\tawd \colon N_*(E\g1_2) \otimes N_*(X) \to N_*(X) \otimes N_*(X)$ that is used to define $\smallsmile_n$ products.
Then in Subsection 2.3 we compute the $\smallsmile_n$ products in $N^*(B\Sigma_2)$ and $N^*(E\Sigma_2)$, using some simple combinatorial facts about counting certain kinds of partitions of integers.\footnote{The proofs of the combinatorial facts are deferred to an appendix.} In particular, we prove Theorem~\ref{binom1} of the Introduction.

\subsection{Alexander-Whitney and Steenrod Maps}

\paragraph{2.1.1} We will make heavy use of the classical Alexander-Whitney map
\begin{equation*}
AW \colon N_*(X \times Y) \to N_*(X) \otimes N_*(Y).
\end{equation*}
Simplices of dimension $n$ in a product simplicial set $X \times Y$ are given by pairs $ (X \times Y)_n = X_n \times Y_n$.
The map $AW$ is defined universally on a basic product of simplices $\Delta^n \times \Delta^n$ by
$$AW \big( (0, 1, \ldots, n), (0, 1, \ldots, n) \big) = \sum_{i=0}^n (0, 1, \ldots, i) \otimes (i, i+1, \ldots, n).$$
On a general pair of $n$-simplices $(u,v)$, this yields by naturality the usual sum of front faces of $u$ tensor back faces of $v$.\\

We will call by the name $\awd \colon N_*(X) \to N_*(X) \otimes N_*(X)$ the map which is a chain approximation of the diagonal given by the composition $AW \circ \Delta_*$, where $\Delta_*$ is the chain map associated to the diagonal map of simplicial sets $\Delta \colon X \to X \times X$, $\Delta u = (u, u)$.
The cochain dual of $\awd$ defines the cochain cup product $N^*(X) \otimes N^*(X) \to N^*(X)$.\\

The map $AW$ for products of spaces is associative, hence unambiguously defied for any number of factors.
It is also natural in any number of factors.
The diagonal approximation map $\awd$ for single spaces is natural and (co)associative.

\paragraph{2.1.2} For any space $X$ there is an immensely important enhanced $AW$ diagonal approximation chain map $\tawd \colon N_*(E\g1_2) \otimes N_*(X) \to N_*(X) \otimes N_*(X)$, which is a $\g1_2$-equivariant chain map of degree 0.
Here $T \in \g1_2$ acts on $N_*(E\g1_2)$ in the obvious way, fixes the copy of $N_*(X)$ in the domain, and switches the two copies of $N_*(X) $ in the range.
The map $\tawd$ was constructed by Steenrod using higher homotopies between the diagonal approximations $\awd$ and $T \awd$, \cite{23steenrodproducts}.
The map $(\tawd)_0 \colon \tx_0 \otimes N_*(X) \to N_*(X) \otimes N_*(X)$ is a chain map, (since $\tx_0$ is a cycle), and identifies with $\awd$.
The enhanced diagonal $\tawd$ is natural for maps $X \to Y$.
A precise construction of $\tawd$ is given in Subsection 2.2 below.

\paragraph{2.1.3} For $n > 0$, the cochain dual of the map $(\tawd)_n \colon \tx_n \otimes N_*(X) \to N_*(X) \otimes N_*(X)$ encodes the higher $\smallsmile_n$ product.
To be precise, given cochains $\alpha, \beta$ of degrees $-p, -q$ and a simplex $u$ of dimension $p+q -n$, one has
\begin{equation*}
\langle \alpha \smallsmile_n \beta, \, u) \rangle =
\langle \alpha \otimes \beta, \ \tawd(\tx_n \otimes u) \rangle.
\end{equation*}
Since $\tawd$ is a chain map, it is a cycle of degree 0 in the complex
\begin{equation*}
\mathrm{Hom}(N_*(E\g1_2) \otimes N_*(X),\ N_*(X) \otimes N_*(X)).
\end{equation*}
This means $0 = \tawd \circ \partial + \partial \circ \tawd$.
Pursuing this, one obtains the coboundary formula for the $\smallsmile_n$ operations,
\begin{equation*}
d (\alpha \smallsmile_n \beta) = d\alpha \smallsmile_n \beta + \alpha \smallsmile_n d \beta + \alpha \smallsmile_{n-1} \beta + \beta \smallsmile_{n-1} \alpha.
\end{equation*}
We can ignore signs since we have $\FF_2$ coefficients.

\paragraph{2.1.4} From the coboundary formula, if $\alpha$ and $\beta$ are cocycles then $$d(\alpha \smallsmile_n \beta) = \alpha \smallsmile_{n-1} \beta + \beta \smallsmile_{n-1} \alpha.$$  If $\alpha$ is a cocycle of degree $-i$ then for $0 \leq n \leq i$ the Steenrod Square $Sq^n \alpha = \alpha \smallsmile_{i -n} \alpha$ is a cocycle of degree $-(i + n)$.
Then $Sq^i(\alpha) = \alpha^2$.
If $\beta$ is another cocycle of degree $-i$ then $$Sq^n(\alpha + \beta) = (\alpha + \beta) \smallsmile_{i-n} (\alpha + \beta) = Sq^n(\alpha) + Sq^n(\beta) + d (\alpha \smallsmile_{i - n + 1} \beta),$$ so the $Sq^n$ are linear operations on cohomology classes.
It is easy to see from the direct combinatorial construction of $\tawd$ given in the next section that $Sq^0(\alpha) = \alpha \smallsmile_i \alpha = \alpha$.\\

Another property of Steenrod Squares that follows easily from Steenrod's direct definition is the commutativity of Squares with cohomology suspension.
The proof of this result seems almost awkward when expressed in terms of cohomology, \cite{13lurie}, \cite{21moshertangora}.
In the papers \cite{6brumfielmorgan} and \cite{7brumfielmorgan} we discovered and exploited the following cochain level formula for the integral version of $\smallsmile_n$ products:
\begin{equation*}
s(x \smallsmile_i y) = (-1)^{deg(x)+i+1}sx \smallsmile_{i+1} sy,
\end{equation*}
where $s$ is cochain suspension.
This obviously implies quite a bit more than just the fact that Steenrod Squares commute with cohomology suspension.\footnote{We believe this to be an unnoticed or under-appreciated formula.}

\paragraph{2.1.5} It is easy to compute cup products in $H^*(B\g1_2)$ and $H^*(E\g1_2)$.
For any $EG$, we have the $AW$ maps, which are $G$-equivariant,
\begin{align*}
AW( (g_0, \ldots, g_n) \times (h_0, \ldots , h_n) ) & = \sum_{i = 0} ^ n (g_o, \ldots, g_i) \otimes (h_i, \ldots, h_n), \\
\awd(g_0, \ldots, g_n) & = \sum_{i = 0}^n (g_0, \ldots, g_i) \otimes (g_i, \ldots, g_n).
\end{align*}

Reducing mod the $G$ action gives the $AW$ maps for $BG$.\\

In the case $G = \g1_2$ one can reason directly with the cells $[T, T, \ldots, T]$ of $B\g1_2$.
The result is easily seen to be $\awd(x_p) = \sum_{i+j = p} \, x_i \otimes x_j$.
Applying these formulae to dual cochains in $B\g1_2$, we get the cup product formula
$$t_i \smallsmile t_j = t_{i+j} \in N^*(B\g1_2).$$
Thus the ring $N^*(B\g1_2) = \FF_2[t]$, a polynomial ring on one generator $t = t_1$ of degree $-1$.
This ring is also the cohomology ring $H^*(B\g1_2)$, since the differential is 0.\\

The $AW$ formulae also reveal the cup products in $N^*(E\g1_2)$.
Explicitly,
\begin{align*}
\tilde{t}_p \smallsmile \tilde{t}_q & = \tilde{t}_{p+q} \text{ and } \tilde{t}_p \smallsmile T \tilde{t}_q = 0 \text{ if } p \text{ is even}, \\
\tilde{t}_p \smallsmile \tilde{t}_q & = 0 \text{ and } \tilde{t}_p \smallsmile T\tilde{t}_q = \tilde{t}_{p+q} \text{ if } p \text{ is odd}.
\end{align*}

The other products in $N^*(E\g1_2)$ are determined by $T$-equivariance.

\subsection{Explicit Definition of $\tawd$}

\paragraph{2.2.1} We now recall Steenrod's explicit cochain formulae for $\smallsmile_n$ products, \cite{23steenrodproducts}.
In fact, we will define Steenrod's map $\tawd \colon N_*(E\g1_2) \otimes N_*(X) \to N_*(X) \otimes N_*(X)$.
A simplex $u \in X_N$ can be viewed as a simplicial map $\Delta^N \to X$, so by naturality it suffices to work on a simplex, $\Delta^N = (0, 1, \ldots, N)$.
Then $\tawd(\tx_n \otimes \Delta^N) \in N_*(\Delta^N) \otimes N_*(\Delta^N)$ is a sum of tensor products of various faces of $\Delta^N$.
Subsets $I \subset (0, 1, \ldots , N)$ name the faces of $\Delta^N$.
The sum we want is indexed by a set, ${\rm Diagrams }(N)$, of diagrams consisting of two rows of non-empty intervals of the vertices of $\Delta^N$.
The total number of intervals is $n+2$, thus $I_1, I_2, \ldots, I_{n+2}$.
Every vertex of $\Delta^N$ is in at least one interval.
For each $1 \leq j \leq n+1$, the final vertex of $I_j$ is the initial vertex of $I_{j+1}$, and this describes the only overlaps of the intervals.
We require each interior interval $I_2, I_3, \ldots I_{n+1}$ to contain more than one vertex.\\

We alternate these intervals, with $I_1, I_3, I_5, \ldots$ on the first row and $I_2, I_4, \ldots$ on the second row.
Visualize the intervals by inserting $n+1$ separating bars between vertices of $\Delta^N$, then repeating the vertex after each bar.
Thus $$(I_1 | I_2 | \ldots | I_{n+2}) = ((0, \ldots, k_1) | (k_1, \ldots, k_2 ) | \ldots | (k_{n+1}, \ldots, N)).$$ 

\paragraph{2.2.2} Here is Steenrod's formula.
$$\tawd(\tilde x_n \otimes \Delta^N)\ =\!\! \sum_{{\rm Diagrams }(N)} (I_{odd} \otimes I_{even}), \text{ where } I_{odd} = \bigsqcup_j I_{2j-1} \text{ and } I_{even} = \bigsqcup_j I_{2j}.$$
In terms of a simplex $u \colon \Delta^N \to X$, with faces denoted $u(I)$, the formula is $$\tawd (\tilde x_n \otimes u) \ =\!\! \sum_{{\rm Diagrams }(N)} u(I_{odd}) \otimes u(I_{even}) \in N_*(X) \otimes N_*(X).$$

We extend $\tawd$ equivariantly, defining $\tawd(T \tilde x_n \otimes u)$ by switching the order of the tensor product factors in $ \tawd (\tilde x_n \otimes u)$.
Given the intervals $I_j$, we could also form a diagram by putting the $I_1, I_3, \ldots$ on the second row and the $I_2, I_4, \ldots$ on the first row.
So the equivariance amounts to a sum over diagrams vs a sum over inverted diagrams.
Of course, it is by no means obvious that $\tawd$ is a chain map.
But Steenrod proved that it is.\\

Now, given cochains $\alpha, \beta \in N^*(X)$ of degrees $-i, -j$ with $i+j-n = N$, and a simplex $u$ of dimension $N$, Steenrod's definition of the $\smallsmile_n$ product becomes, by duality, the following coface formula:
\begin{equation*}
\langle \alpha \smallsmile_n \beta, u \rangle \ = \!\!
\sum_{{\rm Diagrams }(N)} \langle \alpha, u(I_{odd})\rangle \langle \beta, u(I_{even})\rangle.
\end{equation*}

\paragraph{2.2.3} \textbf{REMARK.} For $1 \leq k \leq n+2$, define $|I_k|$ to be the number of vertices in $I_k$.
For $2 \leq k \leq n+1$ we have $|I_k| \geq 2$.
Possibly $|I_1| = 1$ and/or $|I_{n+2}| = 1$.
Of course a diagram contributes 0 for a pair of cochains of degrees $-i, -j$ unless $i+1 = |I_1| + |I_3| + \cdots$ and $j+1 = |I_2| + |I_4| +\cdots$.
We also notice that if $n = 2m$ is even, then the intervals are $I_1, I_3, \ldots, I_{2m+1}$ and $I_2, I_4, \ldots, I_{2m+2}$.
So there are $m+1$ intervals on the first row and $m+1$  intervals on the second row.
If $n = 2m+1$ is odd then the intervals are $I_1, I_3, \ldots, I_{2m+3}$ and $I_2, I_4, \ldots, I_{2m+2}$.
So there are $m+2$ intervals on the first row and $m+1$ intervals on the second row.

\subsection{Cup$_n$ Products in $E\g1_2$ and $B\g1_2$}

\paragraph{2.3.1} We will now embark on a calculation of {\it all} $\smallsmile_n$ products in $N^*(B\g1_2)$ and $N^*(E\g1_2)$.
The computations will make use of some standard combinatorial formulae for counting ordered partitions of positive integers.\\

We fix $i,j,n$ and $N = i+j-n$.
The only non-zero cochains and chains are the dual pairs $t^i = x_i^*$ and $ t^j = x_j^*$.
We distinguish the cases $n = 2m$ even and $n = 2m+1$ odd.
By Remarks 1.4.2 and 2.2.3, in the even case $n = 2m$ the only diagrams that give non-zero evaluations on $t^i \otimes t^j$ are the diagrams for which $i + 1 = |I_1| + |I_3| + \cdots + |I_{2m+1}|$ is a positive partition, with all terms other than $|I_1|$ even, and also for which $j+1 = |I_2| + \cdots + |I_{2m+2}|$ is a positive partition with all terms other than $|I_{2m+2}|$ even.\\

The point here is that on the second row, where $t^j$ will be evaluated, the number of vertices deleted between intervals $I_{2k}$ and $I_{2k+2}$, $k \geq 1$, is $|I_{2k+1}| - 2$.
Thus by Remark 1.4.2, the $j$-face of $x_N$ named $x_N(I_{even})$ is degenerate unless all $|I_{2k+1}|$ are even.
Similarly, on the first row, the $i$-face of $x_N$ named $x_N(I_{odd})$ is degenerate unless all $|I_{2k}|$, $k \leq m$, are even.\\

In the same way by Remarks 1.4.2 and 2.2.3, in the odd case $n = 2m+1$ the only diagrams that give non-zero evaluations are the diagrams for which $i + 1 = |I_1| + |I_3| + \cdots + |I_{2m+3}|$ is a positive partition for which all but the first and last terms are even, and also for which $j+1 = |I_2| + \cdots + |I_{2m+2}|$ is a positive partition with all terms even.

\paragraph{2.3.2} We prove the following facts in an appendix.\\

\textbf{COMBINATORIAL FACT 1.} Let $n = 2m+1$.
The number mod 2 of positive ordered partitions of $i+1$ consisting of $m+2$ terms with all but the first and last even, is the binomial coefficient $\binom {i} {n}$ mod 2.
The number mod 2 of positive ordered partitions of $j+1$ consisting of $m+1$ even terms is the binomial coefficient $\binom {j} {n}$ mod 2.\\

\textbf{COMBINATORIAL FACT 2.} Let $n = 2m$.
The number mod 2 of positive ordered partitions of $i+1$ consisting of $m+1$ summands, all but the first even, is the binomial coefficient $\binom {i} {n}$ mod 2.
Thus also the number mod 2 of positive ordered partitions of $j+1$ consisting of $m+1$ summands, all but the last even, is the binomial coefficient $\binom{j} {n}$ mod 2.

\paragraph{2.3.3} The key consequence of the combinatorial facts and the discussion preceding those statements is the following, which repeats Theorem~\ref{binom1} of the introduction.

\begin{thm} \label{t:binomial repeated}
	For $\tawd \colon N_*(E\g1_2) \otimes N_*(B\g1_2) \to N_*(B\g1_2) \otimes N_*(B\g1_2)$, we have the chain formula $$\tawd (\tx_n \otimes x_k) \, = \!\! \sum_{i+j = k+n} \binom{i}{n} \binom{j}{n} x_i \otimes x_j.$$  The equivalent cochain formula is $$ t^i \smallsmile_n t^j = \binom{i}{n} \binom {j} {n} t^{i+j- n}.$$
\end{thm}

\paragraph{2.3.4} \textbf{REMARK.} The Steenrod Squares are defined for cocycles $\alpha$ of degree $-i$ by $Sq^m(\alpha) = \alpha \smallsmile_{i-m} \alpha$.
Thus we have computed the Steenrod Squares in $N^*(B\g1_2)$ by a direct combinatorial method.
The result is $Sq^m(t^i) = \binom{i} {m} t^{i+m},$ using that $\binom{i}{m} = \binom{i} {i-m}$ and $\binom{i}{m}^2 \equiv \binom{i}{m}$ mod 2.
The usual proof of the formula for Steenrod Squares in real projective space uses the Cartan formula.
We have avoided the Cartan formula, and, moreover we have computed all $\smallsmile_n$ products in a model of real projective space.

\paragraph{2.3.5} We can also calculate the $\smallsmile_n$ products in $N^*(E\g1_2)$.
The method is the same, based on Remarks 1.4.2 and 2.2.3, and the combinatorics of counting partitions.
A new wrinkle arises in the $E\g1_2$ case dealing with the first interval $I_1$ in the diagrams for computing $\smallsmile_n$.
In evaluating a $\smallsmile_n$ product diagram on a cell $(1, T, 1 , \dots)$, the parity of $|I_1|$ determines whether the face to be evaluated on the second row begins with $1$ or $T$.
If $|I_1|$ is even, the second row begins with $T$.
But we know how to count the appropriate partitions mod 2 when $|I_1|$ is even, and also when $|I_1|$ is arbitrary.
So the case $|I_1|$ odd will be the difference, or sum, of those numbers.
Here is the result.

\begin{thm} \label{t:2.2}
	For $\tawd \colon N_*(E\g1_2) \otimes N_*(E\g1_2) \to N_*(E\g1_2) \otimes N_*(E\g1_2)$ we have the chain formula
	\begin{equation*}
	\tawd(\tx_n \otimes \tx_k)\ =\!
	\sum_{i+j = k+n} \binom{i+1}{n+1} \binom{j}{n} \ \tx_i \otimes \tx_j \ + \
	\binom{i}{n+1}\binom{j}{n} \ \tx_i \otimes T \tx_j.
	\end{equation*}
	
	A formula that includes all four evaluation cases for $N_*(E\g1_2)$ is given for $b, a \in \{0, 1\}$ by 
	\begin{equation*}
	\tawd(T^b \tx_n \otimes T^a\tx_k)\ = \sum_{\epsilon = 0, 1}\ \ \sum_{i+j = k+n} c_{n,k,i}^\epsilon \ S^b(T^{a} \tx_i \otimes T^{a+\epsilon}\tx_j ),
	\end{equation*}
	where
	\begin{equation*}
	c_{n,k,i}^0 = \binom{i+1}{n+1}\binom{j}{n}
	\quad \text{and} \quad
	c^1_{n,k,i} = \binom{i}{n+1}\binom{j}{n},
	\end{equation*}
	and where $S$ switches the tensor factors.
\end{thm}

\begin{proof}
	The first statement will be proved in the combinatorial appendix.
	The remaining evaluations are determined by equivariance.\\

	For example, applying $T$ to the second variable $\tx_k$, one uses naturality of $\tawd \colon N_*(E\g1_2) \otimes N_*(X) \to N_*(X) \otimes N_*(X)$ in $X$.
	Applying $T$ to the first variable $\tilde x_n$, one uses the equivariance that applies the operator $S$ interchanging the two factors in the range.
\end{proof}

Note that we have
\begin{equation*}
\binom{i}{n+1} + \binom{i+1}{n+1} = \binom{i}{n+1} + \binom{i}{n} + \binom{i}{n+1} \equiv \binom{i}{n} \pmod{2},
\end{equation*}
in agreement with Theorem~\ref{t:binomial repeated}.

\section{The Chain Maps Phi and Psi}

The main results Theorem~\ref{Adem1} and Theorem~\ref{Adem2} are proved in this section, modulo explaining the operad method for extending Steenrod's $\smallsmile_n$ operations to multi-variable cochain operations.\footnote{The operad discussion is carried out in Section 4.} The key steps in the proofs of Theorem~\ref{Adem1} and Theorem~\ref{Adem2} amount to constructing some chain homotopies with target various complexes $N_*(EG)$, arising from certain equivariant chain maps.
The needed equivariant chain maps are constructed in Subsections 3.1 and 3.2.
The maps they induce on chain complexes are computed in Subsections 3.3 and 3.4.
The formulae for these chain maps in Corollary~\ref{cor3.2} and Corollary~\ref{cor3.3} look somewhat complicated, but do not really involve anything more than the classical Alexander-Whitney diagonal map and the computation of the $\smallsmile_n$ products in $N_*(E\Sigma_2)$.\\

From Subsection 1.3, certain pairs of equivariant chain maps are connected by canonical chain homotopies.
Theorem~\ref{t:3.4}, Corollary~\ref{cor3.5} and Corollary~\ref{cor3.6} in Subsection 3.4 summarize these chain homotopy formulae in our special cases.
Then in Subsection 3.5 we summarize how these results imply Theorem~\ref{Adem1} and Theorem~\ref{Adem2}, modulo explaining how operad methods produce actions of various complexes $N_*(EG)$ on multi-tensors of cochains.

\subsection{Some Dihedral Group Actions}\label{sect3.1}

\paragraph{3.1.1} Let $D_8$ be the dihedral group of order 8, with generators $a, b, c$.
These generators are all of order 2;  $[b,c] = 1$; $ab = ca$; and $ac = ba$.
These are the relations that hold for the inclusion $D_8 \subset \Sigma_4$ given by $b = (12)$, $c = (34)$, $a = (13)(24)$.
Note that $V_4 \subset D_8$ is the subgroup generated by $\{a, bc\}$.\\
 
We define a left action of $D_8$ on $ N_*(E\Sigma_2)\otimes N_*(E\Sigma_2)\otimes N_*(E\Sigma_2)$ as follows.
Let $T$ be the generator of $\Sigma_2$ with its natural action on $ N_*(E\Sigma_2)$.
Then
\begin{align*}
a(x \otimes y \otimes z) & = Tx \otimes z \otimes y, \\
b(x \otimes y \otimes z) & = x \otimes Ty \otimes z, \\
c(x \otimes y \otimes z) & = x \otimes y \otimes Tz.
\end{align*}
One checks easily that these equations define a group action.
Since $N_*(E\Sigma_2)$ is an acyclic chain complex of free $\FF_2[\Sigma_2]$-modules, it is easy to see that this action makes $N_*(E\Sigma_2) \otimes N_*(E\Sigma_2) \otimes N_*(E\Sigma_2)$ an acyclic chain complex of free $\FF_2[D_8]$-modules.

\paragraph{3.1.2} Next, since $EZ$ is an associative operation on products of spaces, there is a well-defined Eilenberg-Zilber map
\begin{equation*}
EZ\colon N_*(E\Sigma_2)\otimes N_*(E\Sigma_2)\otimes N_*(E\Sigma_2) \to N_*(E\Sigma_2 \times E\Sigma_2 \times E\Sigma_2).
\end{equation*}
We also give the range of this map a $D_8$ action.
In fact, $D_8$ acts on the space $E\Sigma_2\times E\Sigma_2\times E\Sigma_2$ by the same formulae on cell triples $(x, y, z)$ of the same degree as the above formulae on basic tensor triples.
The space $E\Sigma_2 \times E\Sigma_2 \times E\Sigma_2$ is then a contractible, free $D_8$ space.
The map $EZ$ is equivariant with respect to the two actions.
The Alexander-Whitney map
\begin{equation*}
AW \colon  N_*(E\Sigma_2\times E\Sigma_2\times E\Sigma_2) \to N_*(E\Sigma_2)\otimes N_*(E\Sigma_2)\otimes N_*(E\Sigma_2)
\end{equation*}
is equivariant for the actions of $b$ and $c$, but not for the action of $a$.

\paragraph{3.1.3} The action of $D_8$ on the product space $E\Sigma_2\times E\Sigma_2\times E\Sigma_2$ is determined by the action on the vertices $\g1_2 \times \g1_2 \times \g1_2$, extended coordinate-wise to
\begin{equation*}
(E\Sigma_2\times E\Sigma_2\times E\Sigma_2)_n \simeq (\Sigma_2\times \Sigma_2\times \Sigma_2)^{n+1}.
\end{equation*}
The free left action of $D_8$ on vertices identifies the group $D_8$ with the set $\g1_2 \times \g1_2 \times \g1_2$, viewed as the $D_8$-orbit of $(1, 1, 1)$.
This identification determines a multiplication on $\g1_2 \times \g1_2 \times \g1_2$, making this set of triples into a group isomorphic to $D_8$.
Moreover, with this group structure and the coordinatewise action of $D_8$ on $n$-simplices of $E(\g1_2 \times \g1_2 \times \g1_2) \simeq E\Sigma_2\times E\Sigma_2\times E\Sigma_2$, we see that in fact we have defined an isomorphism of simplicial sets $E\Sigma_2 \times E\Sigma_2 \times E\Sigma_2 \simeq ED_8$, as free $D_8$ complexes.\\

The actual formula for the induced group product on triples $\g1_2 \times \g1_2 \times \g1_2$ is somewhat tricky.
We can name the triples $(T^{\epsilon_1}, T^{\epsilon_2}, T^{\epsilon_3})$, with $\epsilon_j \in \{0, 1\}$.
Such a triple is identified with the element $c^{\epsilon_3} b^{\epsilon_2}a^{\epsilon_1} \in D_8$.
This is true because evaluating that $D_8$ element on $(1,1,1)$ indeed yields $(T^{\epsilon_1}, T^{\epsilon_2}, T^{\epsilon_3})$.\\

To multiply triples, one simply computes products $c^{\delta_3}b^{\delta_2}a^{\delta_1} c^{\epsilon_3} b^{\epsilon_2} a^{\epsilon_1} \in D_8$, using the relations $ac = ba, ab = ca, bc = cb$ to  move $a^{\delta_1}$ across $c^{\epsilon_3}b^{\epsilon_2}$ and to commute $b$'s and $c$'s.
The result will then have the form $c^{\gamma_3}b^{\gamma_2}a^{\gamma_1},$ which translates to the product of triples.\footnote{This product on triples can be viewed as a semi-direct product multiplication on $D_8 = \g1_2 \ltimes (\g1_2 \times \g1_2)$, with the subgroup $\langle a \rangle = \g1_2$ acting by conjugation on the normal subgroup $\langle b, c\rangle = \g1_2 \times \g1_2$.} 

\paragraph{3.1.4} A much simpler discussion applies to $\Sigma_2 \times \Sigma_2$ acting on $N_*(E\Sigma_2) \otimes N_*(E\Sigma_2)$ by the tensor product of the natural action of the group factors on the tensor factors, and on $ N_*(E\Sigma_2 \times E\Sigma_2)$ by the product action on the product space.
Both the $EZ$ and $AW$ maps are equivariant in this case.\\

Notice that the product $E\Sigma_2\times E\Sigma_2$ is naturally identified with $E(\Sigma_2\times \Sigma_2)$.
Thus, we can view the Alexander-Whitney map as a map
\begin{equation*}
AW\colon N_*(E(\Sigma_2\times\Sigma_2)) \to N_*(E\Sigma_2)\otimes N_*(E\Sigma_2).
\end{equation*}
It is equivariant with respect to the natural $\Sigma_2 \times \Sigma_2$-actions.

\subsection{The Map Phi} \label{sect3.2}
 
\paragraph{3.2.1} We have the Alexander-Whitney diagonal map
\begin{equation*}
\awd \colon N_*(E\Sigma_2)\to N_*(E\Sigma_2)\otimes N_*(E\Sigma_2)
\end{equation*}
and the map of Steenrod
\begin{equation*}
\tawd\colon N_*(E\Sigma_2)\otimes N_*(E\Sigma_2)\to N_*(E\Sigma_2)\otimes N_*(E\Sigma_2).
\end{equation*}
We form the composition below, $\Phi = (\Id \otimes \, \tawd) \circ (\awd \otimes \Id)$,
\begin{equation*}
\begin{tikzcd}
N_*(E\Sigma_2) \otimes N_*(E\Sigma_2) \arrow[r, "\awd \otimes \Id"] &[10pt]
N_*(E\Sigma_2 )\otimes N_*(E\Sigma_2) \otimes N_*(E\Sigma_2) \arrow[r, "\Id \otimes \tawd"] &[6pt]
N_*(E\Sigma_2)\otimes N_*(E\Sigma_2) \otimes N_*(E\Sigma_2).
\end{tikzcd}
\end{equation*}
 
\paragraph{3.2.2} The following result will be important for the later construction of certain chain homotopies.
\begin{thm} \label{t:3.1}
	The composition $\Phi$ is equivariant with respect to the embedding
	\begin{equation*}
	\Sigma_2\times \Sigma_2\to D_8
	\end{equation*}
	that sends the generator of the first factor to $a$ and the generator of the second factor to $bc$.
	(This embedding coincides with the inclusion $V_4 \subset D_8$ as subgroups of $\Sigma_4$.)
\end{thm}

\begin{proof}
	This follows for the involution in the first factor immediately from the fact that $\awd (Tx)= (T \otimes T)(\awd (x))$ and $\tawd(Ty \otimes z) = S(\tawd(y \otimes z))$, where $S$ denotes the switch of tensor factors.
	It follows for the involution in the second factor from the fact that $\tawd(x \otimes Ty) = (T \otimes T)(\tawd(x \otimes y))$.
\end{proof}

\paragraph{3.2.3} As an immediate consequence from the definition $\widetilde x_q = (1, T, 1, \ldots,T^q)$ and the definition of $\awd$ we have the following.
\begin{lem*}
	The Alexander-Whitney diagonal approximation for the simplicial set $N_*(E\Sigma_2)$ is given by
	\begin{equation*}
	\awd (\widetilde x_q)=\sum_{i=0}^q\widetilde x_i\otimes T^i\widetilde x_{q-i}
	\end{equation*}
	and 
	\begin{equation*}
	\awd (T\widetilde x_q)=\sum_{i=0}^qT\widetilde x_i\otimes T^{i+1}\widetilde x_{q-i}.
	\end{equation*}
\end{lem*}
 
\paragraph{3.2.4} We recall from Theorem~\ref{t:2.2} that the map 
\begin{equation*}
\tawd \colon N_*(E\Sigma_2) \otimes N_*(E\Sigma_2) \to N_*(E\Sigma_2)\otimes  N_*(E\Sigma_2)
\end{equation*}
is determined by $\g1_2 \times \g1_2$ equivariance and the formula
\begin{equation*}
\tawd(\widetilde x_n \otimes \widetilde x_k) =
\sum_{\epsilon = 0,1} \sum_{i+j = k+n} c_{n,k,i}^\epsilon \, \widetilde x_i \otimes T^\epsilon \widetilde x_j,
\end{equation*}
where the coefficients $c_{n,k,i}^\epsilon$ are given in Theorem~\ref{t:2.2}.
Specifically, with $j = n+k-i$,
\begin{equation*}
c^0_{n,k,i} = \binom{i+1}{n+1} \binom{j}{n}\ \ \ \ \text{and}\ \ \ c_{n,k,i}^1 = \binom{i}{n+1}\binom{j}{n}.
\end{equation*}
 
Let us write out explicitly the composition $\Phi =(\Id \otimes \tawd) \circ (\awd \otimes \Id)$.
Thus, for $a \in \{0, 1\}$, we have from Lemma 3.2.3 and the full version of Theorem~\ref{t:2.2} that includes the equivariance
\begin{equation} \label{Phi1}
\begin{split}
\Phi(\widetilde x_q\otimes T^a\widetilde x_p) & = \sum_{i=0}^q \tilde x_i \otimes \tawd(T^i \tilde x_{q-i} \otimes T^a \tilde x_p) \\ & =
\sum_{i=0}^q \widetilde x_i \otimes \left(\sum_{\epsilon=0,1} \sum_{j=0}^{p+q-i}c_{q-i,p,j}^\epsilon S^i(T^a\widetilde x_j\otimes T^{a+\epsilon} \widetilde x_{p+q-i-j})\right), 
\end{split}
\end{equation}
where $S$ is the switch of factors.
We also have
\begin{equation} \label{Phi2}
\begin{split}
\Phi(T \widetilde x_q \otimes T^a \widetilde x_p) & =
\sum_{i=0}^q T \tilde x_i \otimes \tawd(T^{i+1} \tilde x_{q-i} \otimes T^a \tilde x_p) \\ &=
\sum_{i=0}^q T\widetilde x_i \otimes \left(\sum_{\epsilon = 0,1} \sum_{j=0}^{p+q-i}c_{q-i,p,j}^\epsilon \ S^{i+1}(T^a\widetilde x_j\otimes T^{a+\epsilon}\widetilde x_{p+q-i-j})\right).
\end{split}
\end{equation}

In parsing these formulae, the subscripts $n, k, i, j$ in the expression for $\tawd(T^b \widetilde x_n \otimes T^a \widetilde x_k)$ of Theorem~\ref{t:2.2} become subscripts $q-i, p, j, p+q-i-j$ in the expressions for $\tawd(T^b\widetilde x_{q-i}\otimes T^a\widetilde x_p)$ that occur in identities \eqref{Phi1} and \eqref{Phi2}.
 
\subsection{Explicit computation of $\bar\Phi$} \label{sect3.3}
 
\paragraph{3.3.1} From the formulae in the previous subsection, along with equivariance, we can deduce formulae for actions of $\Phi$ on certain quotients of the domain and range of $\Phi$.
Consider the map formed from $\Phi$ by first dividing the domain of $\Phi$ by the involution on the second factor of $N_*(E\g1_2) \times N_*(E\g1_2)$, and dividing the range of $\Phi$ by the corresponding action of $bc$.
Follow that by dividing the range of $\Phi$ by the full subgroup $\{b,c\}$,
\begin{equation*}
\bar \Phi \colon N_*(E\Sigma_2)\otimes N_*(B\Sigma_2) \to N_*(E\Sigma_2) \otimes N_*(B\Sigma_2) \otimes N_*(B\Sigma_2).
\end{equation*}
 
From Formula~\eqref{Phi1} of \S{3.2.4}, and since $c^0_{q-i,p,j} + c^1_{q-i,p,j} = \binom{j}{q-i} \binom{p+q-i-j}{q-i}$, we have

\begin{equation} \label{Phi3}
\bar \Phi (\widetilde x_q \otimes x_p)\ = \
\sum_{i=0}^q \widetilde x_i \otimes \left( \sum_{j=0}^{p+q-i}\begin{pmatrix}j \\ q-i \end{pmatrix} \begin{pmatrix} p+q-i-j \\ q-i \end{pmatrix} S^i(x_j \otimes x_{p+q-i-j}) \right).
\end{equation}

In this summation, the binomial coefficient product is 0 unless $0 \leq q-i \leq j \leq p$.
We assume these inequalities going forward.\footnote{The binomial coefficients arose when we were counting diagrams in Section 2.2 related to partitions of integers that were used to compute $\smallsmile_n$ products of cochains. Sometimes there are no diagrams of certain shape that evaluate non-trivially on a tensor product of cochains of given dimensions. You do not need a binomial coefficient formula to count the number of positive partitions of $N$ into $M > N$ summands.}

\paragraph{3.3.2} We will rewrite the sum (3.3) so as to easily distinguish the symmetric and non-symmetric terms in the second two tensor factors of the expression.\\ 
  
We set $\ell =(p-q+i)/2$, an element of $\ZZ[\frac{1}{2}].$  We set $a=p-\ell-j$ so that $a$ is congruent to $\ell$ modulo $\Zee$.
Also, $0 \leq p - j = \ell + a$ and $0 \leq j - q + i = \ell - a$.
Thus $-\ell \leq a \leq\ell$.\\

We have $i = q-p+2\ell, \ j = p-\ell-a, \ q-i = p-2\ell$, and $p+q-i-j = p-\ell+a$. Substitute these values into the sum in Formula~\eqref{Phi3}.

\begin{equation} \label{Phi4}
\bar \Phi(\widetilde x_q\otimes x_p) \ = \!\!\! 
\sum_{\substack{-\ell \le a\le \ell;\ \\ a\equiv \ell \pmod \Zee}}
\begin{pmatrix} p-\ell-a \\ p-2\ell \end{pmatrix}
\begin{pmatrix} p-\ell+a \\ p-2\ell \end{pmatrix}
\widetilde x_{q-p+2\ell} \otimes S^{q-p+2\ell} (x_{p-\ell-a} \otimes x_{p-\ell+a}).
\end{equation}

The sum is finite since $\ell \in \ZZ[\frac{1}{2}]$ and $0 \leq \ell \leq p/2$.
We also must have $0 \leq 2\ell+q-p$.
We can remove the powers of the switching operator $S$ in Formula~\eqref{Phi4} because for fixed $\ell$ there is an $S$-invariant term with $a = 0$ and the other terms occur in pairs with indices $\ell, a$ and $\ell, -a$ whose sum is invariant under $S^{q-p+2\ell}$.\\ 

\begin{cor} \label{cor3.2} 
	In $ N_*(E\Sigma_2)\otimes N_*(B\Sigma_2)\otimes N_*(B\Sigma_2)$ we have
	$$\bar\Phi(\widetilde x_q\otimes x_p) = S_{q,p}+NS_{q,p},$$
	where the symmetric terms are
	$$S_{q,p}=\sum_\ell \begin{pmatrix}p-\ell \\ p-2\ell\end{pmatrix}\widetilde x_{q-p+2\ell}\otimes x_{p-\ell}\otimes x_{p-\ell},$$
	and the non-symmetric terms are
 	\begin{align*}
	NS_{q,p}\ = \!\! \sum_{\substack{-\ell \le a\le \ell;\ \ a\not=0; \\ a\equiv \ell\pmod \Zee}}\begin{pmatrix}p-\ell-a \\ p-2\ell\end{pmatrix}\begin{pmatrix}p-\ell+a \\ p-2\ell\end{pmatrix}\widetilde x_{q-p+2\ell}\otimes x_{p-\ell-a}\otimes x_{p-\ell+a}.
	\end{align*}
	 By symmetry, we also have a formula $\bar\Phi(\widetilde x_p\otimes x_q) = S_{p,q} + NS_{p,q}$.
\end{cor}
 
\paragraph{3.3.3} The map we have constructed 
$$\bar\Phi\colon N_*(E\Sigma_2)\otimes N_*(B\Sigma_2) \to N_*(E\Sigma_2)\otimes N_*(B\Sigma_2)\otimes N_*(B\Sigma_2)$$
is equivariant with respect to the remaining $\Sigma_2$ actions on both sides.
On the range, this involution is $T \otimes S$, where $S$ switches the second and third factors.
$\bar\Phi$ passes to the quotient to give a map that we will also call $\bar \Phi$,
$$\bar \Phi \colon N_*(B\Sigma_2)\otimes N_*(B\Sigma_2) \to \bigl( N_*(E\Sigma_2)\otimes N_*(B\Sigma_2)\otimes N_*(B\Sigma_2)\bigr)_{\g1_2.}$$
This last complex can also be written as $\bigl( N_*(E\Sigma_2)\otimes N_*(E\Sigma_2)\otimes N_*(E\Sigma_2)\bigr)_{D_8},$ where $D_8$ is the dihedral group acting freely as described in \S3.1.1.
As such, the homology of this coinvariant complex is the homology of $BD_8$.
 
\begin{cor} \label{cor3.3}
	We have
	\begin{equation*}
	\bar\Phi( x_q\otimes x_p) = \widehat{S}_{q,p}+\partial \widehat{NS}_{q,p} \in \bigl( N_*(E\Sigma_2)\otimes N_*(B\Sigma_2)\otimes N_*(B\Sigma_2)\bigr)_{\g1_2},
	\end{equation*}
	where
	\begin{align*}
	\widehat{NS}_{q,p} \ = \!\!
	\sum_{\substack{0 < a\le \ell; \\ a\equiv \ell\pmod \Zee}}\begin{pmatrix}p-\ell-a \\ p-2\ell\end{pmatrix}\begin{pmatrix}p-\ell+a \\ p-2\ell\end{pmatrix} [\tilde x_{q-p+2\ell+1} \otimes x_{p-\ell -a} \otimes x_{p-\ell+a}]
	\end{align*}
	and where $$\widehat{S}_{q,p} = \sum_\ell \begin{pmatrix}p-\ell \\ p-2\ell\end{pmatrix}\ [\tilde{x}_{q-p+2\ell}\otimes x_{p-\ell}\otimes x_{p-\ell}].$$
	By symmetry, we also have a formula  $\bar\Phi( x_p\otimes x_q) = \widehat{S}_{p,q} + \partial \widehat{NS}_{p,q}$.
\end{cor}

\begin{proof}
	After dividing by the last $\g1_2$ action, we have in the coinvariant complex $[T\tilde{x} \otimes y \otimes z] = [\tilde{x} \otimes z \otimes y]$.
	So we can combine pairs of terms in Corollary~\ref{cor3.2}.
	Thus with $0 < a$ we have
	\begin{align*}
	\Big[ \tilde{x}_{q-p +2\ell} \otimes (x_{p-\ell-a}\otimes x_{p-\ell+a} + x_{p-\ell+a}\otimes x_{p-\ell-a}) \Big] & =
	\Big[ (\tilde{x}_{q-p +2\ell} +T\tilde{x}_{q-p+2l})\otimes x_{p-\ell-a}\otimes x_{p-\ell+a} \Big] \\ & =
	\partial \Big[ \tilde{x}_{q-p+2\ell+1}\otimes x_{p-\ell-a}\otimes x_{p-\ell+a} \Big].
	\end{align*}
	In parsing this formula and the statement of Corollary~\ref{cor3.3} it is useful to observe that basis elements in the coinvariant complex
	\begin{equation*}
	\Big( N_*(E\Sigma_2)\otimes N_*(B\Sigma_2)\otimes N_*(B\Sigma_2) \Big)_{\g1_2}
	\end{equation*}
	which are non-symmetric in the second two factors have unique names in the form $[T^a \tilde x_r \otimes x_s \otimes x_t]$, with $a \in \{ 0, 1\}$ and $s < t$.
	For symmetric elements, $[\tilde{x}_r \otimes x_s \otimes x_s] = [T\tilde{x}_r \otimes x_s \otimes x_s]$.
	The result is then clear.
\end{proof}

\subsection{The Map Psi}

\paragraph{3.4.1} In \S3.1.3 we identified set-wise $D_8 = \g1_2 \times \g1_2 \times \g1_2$.
We also implicitly translated the product in $D_8$ to a corresponding product of triples.
The group action of $D_8$ on the left of simplices in $E(\g1_2 \times \g1_2 \times \g1_2)$ described in \S3.1.2 then gives an identification of free $D_8$ simplicial sets $ED_8 = E(\g1_2 \times \g1_2 \times \g1_2) .$   We can also identify $E(\g1_2 \times \g1_2 \times \g1_2) $ with $ E\g1_2 \times E\g1_2 \times E\g1_2$.

\paragraph{3.4.2} We form the composition below, $ \Psi = EZ \circ \Phi \circ AW $,
$$N_*(E\g1_2 \times E\g1_2) \xrightarrow{AW} N_*(E\Sigma_2)\otimes N_*(E\Sigma_2) \buildrel \Phi \over\longrightarrow N_*(E\Sigma_2)\otimes N_*(E\Sigma_2)\otimes N_*(E\Sigma_2)$$ 
$$ \xrightarrow{EZ} N_*(E\Sigma_2\times E\Sigma_2\times E\Sigma_2) \simeq N_*(ED_8).$$ 

The last equivalence is from our identification of the simplicial set $ED_8$ with $E\Sigma_2 \times E\Sigma_2 \times E\Sigma_2$ discussed just above.\\

The map $\Psi$ is equivariant with respect to the inclusion $\Sigma_2\times \Sigma_2 \simeq V_4 \subset D_8$ that sends the first generator to $a$ and the second to $bc$.
This holds because $\Phi$ has this equivariance property by Theorem~\ref{t:3.1}, and the $AW$ and $EZ$ maps are also suitably equivariant.

\begin{thm} \label{t:3.4}
	Let $$\iota_* \colon  N_*(E(\Sigma_2\times \Sigma_2))\to N_*(ED_8)$$
	be the inclusion induced by the inclusion $\iota \colon \Sigma_2\times \Sigma_2 \simeq V_4 \subset D_8$.
	Then there is an explicit $\iota$-equivariant chain homotopy
	$J_\Psi$ between $\Psi$ and $\iota_*$, given by the formula 
	\begin{align*} J_\Psi(g_0, g_1, \ldots, g_n) & = \sum_j \iota_*(g_0, \ldots, g_j) * \Psi(g_j, \ldots, g_n) \\
	 & = \sum_j (g_0, \ldots, g_j, \Psi(g_j, \ldots, g_n)).
	\end{align*} 
\end{thm}

\begin{proof}
	The statement is immediate from the equivariance we have established for $\Psi = EZ\circ\Phi\circ AW$ and Theorems~\ref{t:homotopy for classifying spaces} and \ref{t:homotopy for iota and Psi} in \S1.3.1 and \S1.3.2.
	In degree 0, the map $\Psi$ is just the inclusion $\FF_2[\g1_2 \times \g1_2] \to \FF_2[D_8]$, hence the induced map on $H_0$ is the identity, so Theorems~\ref{t:homotopy for classifying spaces} and \ref{t:homotopy for iota and Psi} do apply.
\end{proof}

\paragraph{3.4.3} Now we set
\begin{align*}
\tilde x_q \times \tilde x_p =EZ(\widetilde x_q\otimes \widetilde x_p)\in N_*(E\Sigma_2\times E\Sigma_2)  \nonumber \\
x_q \times x_p = EZ(x_q \otimes x_p) \in N_*(B\g1_2 \times B\g1_2).
\end{align*}
Since $AW \circ EZ = \Id$, we see that $AW(\tilde x_q \times \tilde x_p)=\widetilde x_q\otimes \widetilde x_p$, and hence
\begin{equation*}
\Psi(\tilde x_q \times \tilde x_p) = EZ(\Phi(\widetilde x_q\otimes \widetilde x_p)).
\end{equation*}

An immediate consequence of Theorem~\ref{t:3.4} is the following:

\begin{cor} \label{cor3.5}
	$$(\partial \circ J_\Psi+ J_\Psi \circ \partial)(\tilde x_q \times \tilde x_p) = EZ(\Phi(\widetilde x_q\otimes \widetilde x_p))+ (\tilde x_q \times \tilde x_p) \in N_*(ED_8).$$
	By symmetry we have
	$$(\partial \circ J_\Psi+ J_\Psi \circ \partial)(\tilde x_p \times \tilde x_q) = EZ(\Phi(\widetilde x_p\otimes \widetilde x_q))+ (\tilde x_p \times \tilde x_q) \in N_*(ED_8).$$
\end{cor}

Let $\pi_*\colon N_*(EG)\to N_*(BG)$ be the map induced by the projection to the quotient for $G = V_4 \simeq \Sigma_2 \times \Sigma_2$ and $G = D_8$.
The $V_4$-equivariant homotopy $J_\Psi \colon N_*(EV_4) \to N_{*+1}(ED_8)$ yields a commutative diagram
\begin{equation}
\begin{tikzcd}
N_*(EV_4) \arrow[r, "J_\Psi"] \arrow[d, "\pi_*"] & N_{*+1}(ED_8) \arrow[d, "\pi_*"] \\
N_*(BV_4) \arrow[r, "J_\Psi"] & N_{*+1}(BD_8).\\
\end{tikzcd}
\end{equation}

\begin{cor} \label{cor3.6}
	With
	$\bar{J}_\Psi(x_q \times x_p) = \pi_*J_\Psi(\tilde x_q \times \tilde x_p)$ we have 
	$$\partial \bar{J}_\Psi(x_q \times x_p) = EZ(\bar \Phi(x_q \otimes x_p)) + x_q \times x_p \in N_*(BD_8).$$
	By symmetry we have
	$$\partial \bar{J}_\Psi(x_p \times x_q) = EZ(\bar \Phi(x_p \otimes x_q)) + x_p \times x_q \in N_*(BD_8).$$
\end{cor}

\begin{proof}
	Follows from Corollary~\ref{cor3.6} since $\pi_* \partial(\tilde x_q \times \tilde x_p) = \partial (x_q \times x_p) = 0.$
\end{proof}

\subsection{A Brief Summary}

\paragraph{3.5.1} In \S1.3.3 we gave the general formula for a chain homotopy associated to an inner automorphism of a group.
Specializing to $\Sigma_4$, and referring to the discussion in \S0.2.5, we get $\partial \bar{J}_{(23)}(x_q \times x_p) = x_p \times x_q + x_q \times x_p \in N_*(B\Sigma_4)$.\\

Combining with Corollary~\ref{cor3.6} just above, we then have an equality
\begin{equation} \label{e:3.6}
\partial \bar{J}_\Psi(x_q \times x_p) + \partial \bar{J}_\Psi(x_p \times x_q) + \partial \bar{J}_{(23)}(x_q \times x_p)\ = \
EZ(\bar \Phi(x_q \otimes x_p)) + EZ(\bar \Phi(x_p \otimes x_q)) \in N_*(B\Sigma_4).
\end{equation}
Four of the five terms here are actually in $N_*(BD_8)$.
But recall the point made in \S0.2.2 that we use the same names in $N_*(BD_8) \subset N_*(B\Sigma_4)$.\\

To complete the proofs of Theorems \ref{Adem1} and \ref{Adem2}, which were stated in \S0.2.3 and \S0.2.5 and which the reader should review, we will explain how to `evaluate' every term in Formula \eqref{e:3.6} on a cocycle $\alpha \in N^{-n}(X)$, using an action of elements in $N_*(B\Sigma_4)$ on symmetric tensors of the form $\alpha \otimes \alpha \otimes \alpha \otimes \alpha$.
We will explain such an action in Section 4 in terms of operads.
The sum of the five evaluations of the terms in Formula \eqref{e:3.6} will be 0, which will be seen to be equivalent to Theorem~\ref{Adem2}.
The proof of Theorem~\ref{Adem1} will be a more direct application of Corollary~\ref{cor3.6}.

\paragraph{3.5.2} It turns out that the operad method allows us to suppress the $EZ$ step in the two terms on the right-hand side of Formula \eqref{e:3.6} above, and directly apply $\bar \Phi(x_q \otimes x_p)$ and $\bar \Phi(x_p \otimes x_q)$, which lie in $\bigl( N_*(E\Sigma_2)\otimes N_*(B\Sigma_2)\otimes N_*(B\Sigma_2)\bigr)$, to a symmetric tensor.
These two terms were computed in Corollary~\ref{cor3.2} in \S3.3.2, as sums of triple tensors.
The direct evaluation of $\tilde{x}_r \otimes x_s \otimes x_t$ on a symmetric tensor will follow from Lemma~\ref{l:4.1} in \S4.3.4.
The result is
\begin{equation*}
(\tilde{x}_r \otimes x_s \otimes x_t)(\alpha^{\otimes 4}) = (\alpha \smallsmile_s \alpha) \smallsmile_r (\alpha \smallsmile_t \alpha) = Sq^{n-s}(\alpha) \smallsmile_r Sq^{n-t}(\alpha),
\end{equation*} 
where $deg(\alpha) = -n$.
Note that if $s = t$ then $Sq^{n-s}(\alpha) \smallsmile_r Sq^{n-s}(\alpha) = Sq^{2n-s-r}Sq^{n-s}(\alpha)$.\\

Bringing in the formulae of Corollary~\ref{cor3.3}, which are explicit sums of triple tensors with various coefficients and subscripts (which the reader should review and compare with Theorem~\ref{Adem1}), we find that Corollary 3.6 and Lemma~\ref{l:4.1} imply the formulae stated originally in the introduction as Theorem~\ref{Adem1}.
Specifically, from Corollary~\ref{cor3.3}, the evaluation of $\bar{\Phi}(x_q \otimes x_p)$ on a symmetric tensor $\alpha^{\otimes 4}$ consists of evaluating the symmetric part $\widehat{S}_{q,p}$ of $\bar{\Phi}(x_q \otimes x_p)$, which yields by Lemma~\ref{l:4.1} the sum of iterated Steenrod Squares appearing in Theorem~\ref{Adem1}, and evaluating the non-symmetric part $\partial \widehat{NS}_{q,p}$, which yields the coboundary of the sum of $\smallsmile_r$ products of Squares in Theorem~\ref{Adem1}.
Similarly for $\bar{\Phi}(x_p \otimes x_q)$.
Thus, evaluating the terms in Corollary~\ref{cor3.6} on $\alpha^{\otimes 4}$ completes the proof of Theorem~\ref{Adem1}.

\paragraph{3.5.3} The meaning of the evaluation of the three boundary terms on the left-hand side of Formula \eqref{e:3.6} on a symmetric tensor $\alpha^{\otimes 4}$, producing three coboundary terms in $N^*(X)$, will be explained in \S4.4.4.
Then the vanishing of the sum of the five evaluations of the terms in \eqref{e:3.6}, combined with the discussion of Theorem~\ref{Adem1} in \S3.5.2 just above, very easily translates to the statement of Theorem~\ref{Adem2}.

\paragraph{3.5.4} In a strong sense, our main theorem is really Formula \eqref{e:3.6} in \S3.5.1.
For each pair $q > p$, Formula \eqref{e:3.6} can be rewritten as a relation $R(q,p) = 0 \in N_{q+p}(B\Sigma_4)$, where $R(q,p)$ is a sum of five terms.
Explicit formulae for all five of these terms have been given at various points in our paper.
In fact, before applying equivariance to the chain homotopies $J_\Psi$ and $J_{(23)}$, we actually gave formulae for all five terms in the sum 
\begin{align} 
\tilde{R}(q,p)
& = \partial J_\Psi(\tilde x_q \times \tilde x_p) + \partial J_\Psi(\tilde x_p \times \tilde x_q) + \partial J_{(23)}(\tilde x_q \times \tilde x_p) \nonumber \\
& + EZ(\Phi(\widetilde x_q\otimes \widetilde x_p)) + EZ(\Phi(\widetilde x_p\otimes \widetilde x_q)) \in N_{q+p}(E\Sigma_4).
\end{align}

The sum $\tilde{R}(q,p) \in N_{q+p}(E\Sigma_4)$ projects to the sum $R(q,p) \in N_{q+p}(B\Sigma_4)$, and in the operad method to be described in Section 4 it is actually $\tilde{R}(q,p)$ that is directly evaluated on symmetric tensors $\alpha^{\otimes 4}$, with the result depending only on $R(q,p)$.
Conceptually it is not so difficult to write a computer program for calculating the $\tilde{R}(q,p) \in N_{q+p}(E\Sigma_4)$, and we have done so, but the output gets extremely large, even for $q+p \approx 10$.
A positive way to look at this is that each $R(q,p)$ corresponds to {\it infinitely many} different (unstable) Adem relations, obtained as in Theorem~\ref{Adem2} by evaluating $R(q,p)$ on a symmetric tensor $\alpha^{\otimes 4}$, where the cocycle $\alpha$ can have any degree.

\section{Operads} \label{sect4}
At the beginning of this section we introduce some basic terminology about operads.
We describe the symmetric group operad in the category of sets and the endomorphism operads in the category of chain complexes as examples.
In Subsections 4.2 and 4.3 we introduce the Surj operad, $\cS = \{\cS_n\}_{n \geq 1}$, also an operad of chain complexes, and describe how each chain complex $\cS_n$ acts on $n$-fold tensors of cochains on spaces $X$.
The action of $\cS_2$ coincides with Steenrod's two-variable $\smallsmile_n$ products.
Then we explain in Lemma~\ref{l:4.1} how the operad structure exhibits iterations of $\smallsmile_n$ products as part of the action of $\cS_4$.\\

In Subsection 4.4 we introduce the Barratt-Eccles operad $\cE$ with $\cE_n = N_*(E\Sigma_n)$.
The operad $\cE$ acts on multi-tensors of cochains via an operad morphism $TR\colon \cE \to \cS$.
The operad morphism $TR$ is very important in our paper because our key chain homotopies all have values in $\cE$, and we need to push those chain homotopies to $\cS$ to obtain explicit cochain operations and complete the proofs of the main theorems.

\subsection{Introduction} \label{sect4.1}

\paragraph{4.1.1} Recall that a {\em symmetric operad}, \cite{16mayoperads}, $\cP$ in a monoidal category is a collection of objects
$\cP = \{{\mathcal P}_n\}_{n\ge 1}$ of the category together with structure maps
$${\mathcal P}_r\times ({\mathcal P_{s_1}\times\cdots\times{\mathcal P}_{s_r}}) \to {\mathcal P}_{s_1+\cdots+s_r},$$
satisfying composition and symmetry rules for $\Sigma_n$ actions on $\cP_n$, and a unit rule for $\cP_1$.

\paragraph{4.1.2} Let us define the operad in the category of sets determined by the symmetric groups.
We denote an element $\sigma \in \Sigma_n$ by the sequence $(\sigma(1)\ldots \sigma(n))$,
which is the reordering of $(1\ldots n)$ given by applying the permutation.\footnote{For the remainder of the paper we will write permutations in this form, rather than as products of disjoint cycles.}\\
 
The symmetric groups $\{\Sigma_n\}_{n\ge 1}$ form an operad where the operad structure
$$\Sigma_r\times(\Sigma_{s_1}\times \cdots \times \Sigma_{s_r})\to \Sigma_{s_1+\cdots+s_r}$$
is given by dividing the interval $[1,s_1+\ldots+s_r]$ into disjoint blocks of lengths $s_1, \ldots , s_r$, starting from the left, and then first permuting the elements of the $i^{th}$ block among themselves by applying the  conjugation of the element of $\Sigma_{s_i}$ by the unique order preserving map from the $i^{th}$ block to $\{1, \ldots ,s_i\}$.
This produces an automorphism of each block.
The blocks with their new internal orderings are then permuted among themselves, according to the element of $\Sigma_r$.\footnote{One can also first permute the $r$ blocks using $\sigma \in \Sigma_r$, keeping the entries of each original block in their consecutive order. Then apply the $\sigma_i \in \g1_{s_i}$.
To keep straight how each permutation is applied, it helps to pretend that $r, s_1, \ldots, s_r$ are distinct integers.} 
 
\paragraph{4.1.3} As a simple example, consider the operad structure map $\Sigma_2\times (\Sigma_2\times \Sigma_2) \to \g1_4$.
We have the embedding $D_8 \subset \g1_4$, with $a \mapsto (3412),\ b \mapsto (2134),\ c \mapsto (1243)$.
It is easy to see from the block description of the operad structure map that the triple $(T^{\epsilon_1}, (T^{\epsilon_2}, T^{\epsilon_3})) \in \Sigma_2\times (\Sigma_2\times \Sigma_2)$ maps to the permutation $c^{\epsilon_3}b^{\epsilon_2}a^{\epsilon_1} \in D_8 \subset \g1_4$.
Here, $\epsilon_i \in \{0,1\}$.\\

Thus, the operad structure map is a bijection $\Sigma_2\times (\Sigma_2\times \Sigma_2) \simeq D_8$, and, moreover, this bijection coincides with the bijection $D_8 \simeq \Sigma_2\times (\Sigma_2\times \Sigma_2)$ studied in \S3.1.3.
The induced product on the triples $(T^{\epsilon_1}, (T^{\epsilon_2}, T^{\epsilon_3})) \in \Sigma_2\times (\Sigma_2\times \Sigma_2)$ thus also coincides with the product on triples from \S3.1.3.

\paragraph{4.1.4} For a vector space or a chain complex $V$ there is the Endomorphism operad, with
$$End(V)_n={\ \rm Hom}(V^{\otimes n}, V),$$
with the obvious operad structure and action of the symmetric groups.
To give $V$ the structure of an {\em algebra} over a symmetric operad ${\mathcal P}$ is to give a map of symmetric operads ${\mathcal P}\to End(V)$.

\subsection{The Surj Operad $\cS$}

In this subsection and the next we follow McClure and Smith \cite{18mccluresmith}.

\paragraph{4.2.1} We view a function $ \{1, \ldots , r+d\}\to \{1, \ldots ,r\}$ as a sequence of integers $A = (a(1) \ldots a(r+d))$, each in the interval $[1,r]$.
Fix $r$ and consider the $\mathbb{F}_2$-vector space with basis the set of all maps $\{1, \ldots ,r+d\}\to \{1, \ldots ,r\}$.
We form the quotient vector space, denoted $\cS_r(d)$, by setting equal to zero all sequences $(a(1) \ldots a(r+d))$ that are either not surjective functions or have
$a(i)=a(i+1)$ for some $i < r+d$.
We define a chain complex structure on $\cS_r$.
The boundary of a basis element of $\cS_r(d)$ is
$$\partial \big(a(1) \ldots a(r+d)) = \sum_i(a(1) \ldots \widehat{a(i)} \ldots a(r+d)\big) \in \cS_r(d-1).$$
Of course, this means that any terms in the sum that are either non-surjective functions or
have the property that two successive entries are equal are set to zero.
One sees easily that this defines a chain complex denoted $\cS_r$.

\paragraph{4.2.2} There is the obvious left action of $\Sigma_r$ on $\cS_r$ given by post-composition of a function with a permutation of $\{1, \ldots ,r\}$.
The operad structure on $\cS$, which we will not define explicitly, is compatible with these actions.
This means that $\cS$ is a {\em symmetric operad}.
For each $r\ge 1$, the action of $\Sigma_r$ on $\cS_r$ is a free action and it turns out, \cite{18mccluresmith}, that $\cS_r$ is an acyclic resolution of $\mathbb{F}_2$ over $\mathbb{F}_2[\Sigma_r]$.
Hence, the homology of the coinvariant complex $(\cS_r)_{\Sigma_r}$ is identified with $H_*(B\Sigma_r)$.
 
\paragraph{4.2.3} The basis of $\cS_2(n)$ consisting of alternating sequences of 1's and 2's of length $n+2$ matches the basis of $N_n(E\g1_2)$ consisting of alternating sequences of 1's and $T$'s of length $n+1$, by dropping the last entry of a $1, 2 $ sequence.
The $\g1_2$ actions and the boundary formulae also coincide.
Thus we can identify $\cS_2$ and $N_*(E\g1_2)$.
The full operad operations in the Surj operad $\cS$ are quite complicated.
However, on degree zero chains, $ \cS_2 \otimes (\cS_2 \otimes \cS_2) \to \cS_4$ coincides with the map $\FF_2[D_8] \to \FF_2[\g1_4]$ induced by the dihedral group inclusion $D_8 \to \g1_4$ for the symmetric group operad, as described in \S4.1.3.
In all degrees, the axioms for permutation group actions on symmetric operads imply that $$N_*(E\g1_2) \otimes (N_*(E\g1_2) \otimes N_*(E\g1_2)) = \cS_2 \otimes (\cS_2 \otimes \cS_2) \to \cS_4$$ is equivariant for the inclusion $D_8 \to \g1_4$.
 
\subsection{Action of the Surj Operad on Normalized Cochains}

\paragraph{4.3.1} The normalized cochain complex of a simplicial set $X$, $ N^*(X)$, is an algebra over the Surj operad.
That is to say there are chain maps that we will describe in the next subsection
$${\mathcal O}_X\colon \cS_r\otimes \underbrace{ N^*(X)\otimes\cdots\otimes N^*(X)}_{r-{\ \rm times}}\to N^*(X).$$ 
This means that (i) the degree of ${\mathcal O}_X(\zeta\otimes\alpha_1\otimes \cdots\otimes \alpha_r)$
is the sum of the degrees of the $\alpha_i$ plus the degree of $\zeta$, and (ii)
\begin{align*}
d({\mathcal O}_X(\zeta\otimes\alpha_1\otimes\cdots\otimes\alpha_r))=
{\mathcal O}_X \Big\{ \partial(\zeta) \otimes (\alpha_1 \otimes \cdots \otimes \alpha_r) + \sum_i \zeta \otimes (\alpha_1 \otimes \cdots \otimes d\alpha_i \otimes \cdots \otimes \alpha_r) \Big\}.
\end{align*}
Also, the $\mathcal{O}_X$ satisfy certain associativity and equivariance properties.\\

Recall our convention that $N^*(X)$ is negatively graded.
This means that the operation $\cO_X(\zeta)$ decreases the sum of the absolute values of the degrees of the $\alpha_i$ by deg$(\zeta)$.\\

The operations $\cO_X$, natural in $X$, implicitly determine the operad structure maps for $\cS$.
In fact, any finite part of $\cS$ injects into $End(N^*(\Delta^k))$ for large $k$.

\paragraph{4.3.2} We will describe the action of a function $ A = (a(1)\ldots a(n+r)) \in \cS_r(n)$ on multi-tensors of cochains.
For each $k$, we consider a set of $r$-step diagrams, ${\rm Diagrams}(A, k)$, associated to certain collections of $n+r$ subintervals of $[0,k]$.
By this we mean a division of $[0,k]$ into non-empty intervals $I_1,I_2, \ldots ,I_{n+r}$ so that the final point of each interval agrees with the initial point of the next interval.
The interval $I_j$ is said to be at level $a(j)$.
There is one extra condition, which is that for any $j\not= j'$, if the intervals $I_j$ and $I_{j'}$ are at the same level, then they are disjoint.
For a given level $1 \leq \ell \leq r$, set $I(\ell) = \bigsqcup_{a(j) = \ell} I_j $, which we interpret as a face of the simplex $\Delta^k$.\\

The generator $A = (a(1) \dots a(n+r)) \in \cS_r(n)$ acts in the following manner.
Let $\alpha_1 \otimes \cdots \otimes \alpha_r$ be a multi-tensor of cochains of total degree $-(n + k)$.
The operad algebra structure $\cO_X$ will assign to $A \otimes (\alpha_1 \otimes \cdots \otimes \alpha_r)$ a cochain of degree $-k$.
Let $u \in N_k(X)$ be a simplex of dimension $k$, regarded as a simplicial map $u \colon \Delta^k \to X$.
Then
\begin{equation*}
\left \langle \cO_X(A \otimes (\alpha_1 \otimes \cdots \otimes \alpha_r)), \ u \right \rangle \ = \!
\sum_{{\rm Diagrams}(A, k)} \prod_\ell \ \langle \alpha_\ell, \ u(I(\ell)) \rangle.
\end{equation*}

With this definition the element $ (1212 \ldots)$ in $\cS_2(n)$ acts by 
$$\alpha\otimes\beta\mapsto \alpha\smallsmile_n\beta.\footnote{We defined $\smallsmile_n$ as a sum over two-step diagrams in \S\S 2.2.3, 2.2.4, using the terminology {\it odd} and {\it even} for levels 1 and 2.}$$

\paragraph{4.3.3} The coinvariant chain complex of the $\Sigma_r$ action on $\cS_r$ acts on $\Sigma_r$-symmetric cochains in $N^*(X)^ {\otimes r} $ as follows.
Let $\zeta \in \cS_r(n)_{\Sigma_r}$.
Lift $\zeta$ to an element $\widetilde\zeta \in \cS_r(n)$.
The element $\widetilde\zeta$ determines a map
$$ N^*(X)^{\otimes r}\to N^*(X)$$
raising degree by $n$.\footnote{Again recall cochains are negatively graded.}
Another lift will give a different map, but elements in the same $\g1_r$ orbit have the same restriction to the $\Sigma_r$-invariant elements in $ N^*(X)^{\otimes r}$.\\

Thus, we have a well-defined chain map
$$({\mathcal O}_X)_{\Sigma_r}\colon (\cS_r)_{\Sigma_r}\otimes \bigl(\underbrace{ N^*(X)\otimes\cdots\otimes N^*(X)}_{r-{\ \rm times}}\bigr)^{\Sigma_r}\to N^*(X).$$

Restricting even further to symmetric cocycles of the form $\alpha\otimes \alpha\otimes\cdots\otimes \alpha$, 
a cycle in $\cS_r(n)_{\Sigma_r}$ determines a map
from cocycles of degree $k \leq 0$ to cocycles of degree $rk+n$.
The induced operation on cohomology depends only on the homology class of $\zeta$.
In this way $H_n(B\Sigma_r)$ acts as cohomology operations
$$H^k(X)\to H^{rk+n}(X).$$

In the special case when $r=2$ the operation associated to
$$[(\underbrace{1212\ldots}_{(n+2)})] \equiv [(\underbrace{1T1T\ldots}_{(n+1)})] = x_n \in N_n(B\Sigma_2)$$
sends a symmetric cochain $(\alpha\otimes \alpha)\in N^k(X)\otimes N^k(X)$ to $\alpha\smallsmile_n \alpha \in N^{2k+n}(X)$.
In the case when $\alpha$ is a cocycle, the result of this operation is a cocycle representing the cohomology class
$Sq^{|k|-n}([\alpha])$.
 
\paragraph{4.3.4} In particular, we are interested in the operad product
$$\cS_2\otimes (\cS_2\otimes \cS_2)\to \cS_4.$$
This is a map of an acyclic complex with a free action of the dihedral group $D_8$ to an acyclic complex with a free $\Sigma_4$-action.
Of course, $\cS_2$ is identified with $ N_*(E\Sigma_2)$.
This allows us to consider $$\Phi(\widetilde x_q\otimes \widetilde x_p) \in N_*(E\g1_2) \otimes (N_*(E\g1_2) \otimes N_*(E\g1_2))$$ as an element of the domain of this operad product, where the map $\Phi$ was defined in \S3.2.1.
From  \S4.2.3, the operad product is equivariant with respect to the inclusion of $D_8 \subset \Sigma_4$ that sends the elements $b$ and $c$ in the dihedral group to the permutations $(2134)$ and $(1243)$ respectively, and sends the element $a$ in the dihedral group to $(3412)$.
The following is an important observation.

\begin{lem} \label{l:4.1}
	Under the operad composition $\cS_2\otimes (\cS_2\otimes \cS_2) \to \cS_4$, the element
	$\widetilde x_r \otimes (\widetilde x_s\otimes \widetilde x_t)$ maps to an element of $\cS_4$ that acts on a cochain $\alpha_1\otimes\alpha_2\otimes\alpha_3\otimes\alpha_4$ to produce
	$$(\alpha_1\smallsmile_s\alpha_2)\smallsmile_r(\alpha_3\smallsmile_t\alpha_4).$$
\end{lem}

\begin{proof}
	We will not prove this by directly passing through $\cS_4$, but rather by using the operad morphism $\cS \to End(N^*(X))$.
	We have the identification $\cS_2 \simeq N_*(E\g1_2)$ of \S4.2.3, with $\smallsmile_q = (12121...) \leftrightarrow \tilde x_q$.
	The operad morphism sends $\widetilde x_r \otimes (\widetilde x_s\otimes \widetilde x_t)$ to the endomorphism operad element $\smallsmile_r \otimes (\smallsmile_s \otimes \smallsmile_t)$.
	Acting on $N^*(X)^{\otimes 4}$, this gives the composition $$N^*(X)^{\otimes 4} = N^*(X)^{\otimes 2} \otimes N^*(X)^{\otimes 2} \xrightarrow{ \smallsmile_s \otimes \smallsmile_t} N^*(X) \otimes N^*(X) \xrightarrow{\smallsmile_r} N^*(X),$$ which is exactly the claim of the lemma.
\end{proof}
 
\subsection{The Barratt-Eccles Operad $\cE$}

Here we follow the presentation of Berger-Fresse \cite{3bergerfresse}.

\paragraph{4.4.1} The Barratt-Eccles operad is an operad in the category of chain complexes with ${\mathcal E}_r = N_*(E\Sigma_r)$.
The $\cE$ operad structure map $$\cE_r \otimes (\cE_{s_1} \otimes \cdots \otimes \cE_{s_r}) \to \cE_{s_1 + \cdots + s_r}$$ is a composition of the operad structure on symmetric groups and the Eilenberg-Zilber map $$EZ \colon N_*(E\g1_r) \otimes (N_*(E\g1_{s_1} )\otimes \cdots \otimes N_*(E\g1_{s_r})) \to N_*(E\g1_r \times (E\g1_{s_1} \times \cdots \times E\g1_{s_r})).$$ Specifically, the $\cE$ operad structure map is post-composition of $EZ$ with the map of normalized chain complexes induced by the set theoretic map $$\g1_r \times (\g1_{s_1} \times \cdots \times \g1_{s_r}) \to \g1_{s_1 + \cdots + s_r},$$ which is the structure map of the symmetric group operad described in \S4.1.2. \\

As is the case with the Surj operad, the Barratt-Eccles operad is a symmetric operad:
the natural actions of $\Sigma_r$ on ${\mathcal E}_r$ are compatible with the operad structures.
Furthermore, ${\mathcal E}_r$ is a free $\FF_2[\Sigma_r]$ resolution of $\mathbb{F}_2$.

\paragraph{4.4.2} 

Berger-Fresse define an operad morphism $TR \colon \cE \to \cS$, which they call {\em Table Reduction}.
Since Table Reduction is an operad map, for any simplicial set $X$ the normalized cochains
$N^*(X)$ form an algebra over the Barratt-Eccles operad.\\

For completeness in this paper, we include here a definition of the Table Reduction morphism.
Given a basis element $(\sigma_0, \dots, \sigma_n) \in \mathcal E_r(n)$ we define
\begin{equation*}
TR(\sigma_0, \dots, \sigma_n) = \sum_{a} s_{a} \in \cS_r(n)
\end{equation*}
as a sum of surjections
\begin{equation*}
s_{a} \colon \{1, \dots, n+r \} \to \{1, \dots, r\}
\end{equation*}
parametrized by all tuples of integers $a = (a_0, \dots, a_n)$ with each $a_i \geq 1$ and such that $a_0 + \cdots + a_n = n + r$.
For one such tuple $a$ we now describe its associated surjection $s_a$ as a sequence $s_a = (s_a(1), s_a(2), \ldots, s_a(n+r))$.\\

Consider the following ``table" representation of the generator of $\mathcal {E}_r(n)$:
\begin{center}
	\begin{tabular} {c | c c c}
		& $\sigma_0(1)$ & $\cdots$ & $\sigma_0(r)$ \phantom{.} \\
		& $\sigma_1(1)$ & $\cdots$ & $\sigma_1(r)$ \phantom{.} \\
		& $\vdots$ 	& $\ddots$ & $\vdots$ 	\\
		& $\sigma_n(1)$ & $\cdots$ & $\sigma_n(r)$ .
	\end{tabular}
\end{center}

Each line of the table is a permutation of $\{1,2, \ldots, r\}$.
The first $a_0$ entries of the sequence $s_a$ are the first $a_0$ entries of the permutation on the first line of the table.
The first $a_0 -1$ of these values, that is, all except the last one, are then removed from all lower permutation lines of the table.\\

The next $a_1$ entries in the sequence $s_a$ are the first $a_1$ entries remaining in the second line of the table.
The values of all except the last of these are then removed from all lower lines in the table.
The process forming the sequence $s_a$ continues in this way.
The final $a_n$ entries of the sequence will consist of all remaining entries on the last line of the table after the first $n$ steps of the process.\\

For more details on this map we refer the reader to the original treatment in \cite{3bergerfresse}.\\

The component ${\mathcal E}_2$ in the Barratt-Eccles operad is 
$ N_*(E\Sigma_2) = \cS_2$.
On these components of the operads, Table Reduction is the identification of ${\mathcal E}_2$ with $\cS_2$ given above.\\

We shall use the following commutative diagram
\begin{equation*}
\begin{tikzcd}
N_*(E\Sigma_2) \otimes (N_*(E\Sigma_2) \otimes N_*(E\Sigma_2)) \arrow[r] \arrow[d, "{TR \, \otimes (TR \, \otimes TR)}"] & N_*(E\Sigma_4) \arrow[d, "TR"]\\
\cS_2\otimes (\cS_2\otimes \cS_2 ) \arrow[r] & \cS_4,
\end{tikzcd}
\end{equation*}
where the horizontal arrows are the Barratt-Eccles and Surj operad structure maps.
These maps are equivariant with respect to the inclusion $D_8 \subset \g1_4$.\\

The Barratt-Eccles operad map
$$N_*(E\Sigma_2)\otimes ( N_*(E\Sigma_2)\otimes N_*(E\Sigma_2)) \to N_*(E\Sigma_4)$$
is the composition of 
$$EZ\colon N_*(E\Sigma_2)\otimes N_*(E\Sigma_2)\otimes N_*(E\Sigma_2) \to N_*(E(\Sigma_2\times\Sigma_2\times \Sigma_2))$$
with the map
$$N_*(E(\Sigma_2\times\Sigma_2\times \Sigma_2))\to N_*(E\Sigma_4)$$
induced by the group homomorphism $(T,1,1)\mapsto a = (3412);\ \ (1, T, 1)\mapsto b = (2134);\ \ (1, 1, T) \mapsto c = (1243)$, which defines our chosen group inclusion $D_8 \subset \Sigma_4$.\footnote{It is here that we use the identification of the group operation on triples with $D_8$ that we discussed in \S3.1.3 and \S4.1.3.}
 
\paragraph{4.4.3} Recall from \S3.3.3 that $\bar{\Phi}(x_q \otimes x_p) \in (\cS_2\otimes \cS_2\otimes \cS_2)_{D_8}$.
It is exactly the commutativity of the above diagram, and the fact that $EZ$ is part of the $\cE$ operad structure, that explains our remark in \S3.5.2 that we can evaluate $EZ \bar{\Phi}(x_q \otimes x_p) \in N_*(BD_8) \subset N_*(B\g1_4) = (\cE_4)_{\g1_4}$ on a cocycle $\alpha \in N^*(X)$ by directly evaluating $\bar{\Phi}(x_q \otimes x_p) \in (\cS_4)_{\g1_4}$ on $\alpha^{\otimes 4}$.
The actual formula in terms of Steenrod Squares and $\smallsmile_r$ products, then comes from Corollary~\ref{cor3.3} and Lemma~\ref{l:4.1}.
We have now completed one of the last steps in the proofs of the main Theorems \ref{Adem1} and \ref{Adem2}, as summarized in Subsection 3.5.
 
\paragraph{4.4.4} All that remains in the proofs of Theorems \ref{Adem1} and \ref{Adem2} from the summary in Subsection 3.5 is to explain how Table Reduction is used to evaluate the three boundary terms $$ \partial \bar{J}_\Psi(x_q \times x_p),\ \partial \bar{J}_\Psi(x_p \times x_q),\  \partial \bar{J}_{(23)}(x_q \times x_p) \in N_*(B\g1_4)$$ from Formula \eqref{e:3.6} in \S3.5.1 on a cocycle $\alpha.$ Since Table Reduction is a $\g1_4$-equivariant chain map, it induces a chain map of coinvariant complexes $\bar{TR}\colon N_*(B\g1_4) \to (\cS_4)_{\g1_4}.$ Each of the three $\bar J$ terms is an element of $N_*(B\g1_4)$.
We then have three versions of $\bar{TR} ( \partial \ \bar{J} ) = \partial (\bar{TR} \ \bar{J}) \in (\cS_4)_{\g1_4}$.
Given a cocycle $\alpha$, we evaluate each of these three boundaries on the symmetric 4-tensor $\alpha^{\otimes 4}$, using the $\cS$ algebra structure of $N^*(X)$ as described in \S4.3.2 and \S4.3.3.
We thus have
$$(\partial\ \bar{TR} \ \bar{J}) (\alpha^{\otimes 4}) = d\ (\bar{TR} \ \bar{J} (\alpha^{\otimes 4})) \in N^*(X).$$
This last expression, for each of the three $\bar{J}$'s, is what we abbreviated as $d (\bar{J} (\alpha))$ in the original statement of Theorem~\ref{Adem2}.
All we have done here is clarify precisely the operad mechanism alluded to in \S0.1.4 by which chains in $N_*(B\g1_4)$ act on cocycles $\alpha$.
 
\paragraph{4.4.5} We have implemented Table Reduction as part of a computer program, along with all the other ingredients needed to make explicit our coboundary formulae for Adem relations.
Note that as explained in \S4.3.3 and discussed in \S3.5.4, for each of the cochains
\begin{equation*}
J = J_\Psi(\tilde x_q \times \tilde x_p), \ J_\Psi(\tilde x_p \times \tilde x_q), \ J_{(23)}(\tilde x_q \times \tilde x_p) \in N_*(E\Sigma_4),
\end{equation*}
it is really the lifts $TR\ J \in \cS_4$ of the $\bar{TR}\ \bar{J} $ that we evaluate directly on $\alpha^{\otimes 4}$.
For fixed $(q,p)$, the computer output for the $TR\ J \in \cS_4$ is much smaller than the computer output for the $J \in N_*(E\Sigma_4)$.
For a cocycle $\alpha \in H^{-n}(X)$ with fixed small $n$, there is significant further reduction since many surjection generators in $\cS_4$ contribute 0 when evaluated on $\alpha^{\otimes 4}$.
Nonetheless, our coboundary formulae for Adem relations quickly get very lengthy, even for cocycles and relations of low degree.\\
 
This completes our discussion of the main Theorems \ref{Adem1} and \ref{Adem2}.

\section{Combinatorial Appendix}

The first two subsections of the final section of the paper provide proofs of the combinatorial facts about counting partitions that we used to compute the $\smallsmile_n$ operations in $N_*(B\Sigma_2)$ and in $N_*(E\Sigma_2)$.
Then in Subsection 5.3 we provide the details that exploit commutativity of Steenrod operations with suspension to deduce the standard form of Adem relations from the form given in Theorem~\ref{Adem2}.
We find it interesting that this argument proves somewhat more, namely that some of the coboundary terms in our unstable cochain level Adem relations also commute with cochain suspension.
It seems like a good question whether cochain versions of unstable Adem relations can be found that completely commute with cochain suspension.

\subsection{Counting Ordered Partitions}

\paragraph{5.1.1} We begin with the following well-known statement about counting ordered partitions.
The goal is to prove Combinatorial Facts 1 and 2 from \S2.3.2 that were used to evaluate $\smallsmile_n$ products in $N^*(B\g1_2)$ and $N^*(E\g1_2)$.\\

\begin{lem} \label{l:5.1}
	(i): The number of non-negative ordered partitions of an integer $N$ into $M$ summands equals the coefficient of $x^N$ in $(1 + x)^{N+M-1}$, hence is given by the binomial coefficient $\binom {N+M-1}{N}$.
	The number of such partitions is also (more obviously) given by the coefficient of $x^N$ in the expansion of $(1 +x + x^2 + \cdots )^M$.\\
		
	(ii): The number of positive ordered partitions of $N$ into $M$ summands is $\binom{N - 1} {N - M} = \binom{N-1}{M-1}$.
\end{lem}

\begin{proof}
	(i): We give a correspondence between individual terms of degree $N$ in the expansion of $(1+x)^{N+M-1}$ and non-negative ordered partitions of $N$ into $M$ summands.
	The individual terms in the expansion can be viewed as sequences of 0's and 1's of length $N+M-1$, based on whether in each factor $(1+x) $ one selects the $1 = x^0$ or the $x = x^1$.
	Such a sequence consists of blocks of 0's and 1's.
	Remove one 0 from each block of 0's that lies between two blocks of 1's.
	Then the associated ordered partition of $N$ will consist of all the remaining 0's and positive integers corresponding to the number of 1's in each original block of 1's.\\

	For the second statement in part (i), in the expansion of $(1 +x + x^2 + \cdots )^M$, to get a term $x^N$ one must choose some power $x^{n_k}$ in the $k^{th}$ factor so that $n_1 + n_2 + \cdots + n_M = N$.\\
	
	(ii): We pass from non-negative partitions to positive partitions by adding 1 to each summand.
	Thus, the number of positive ordered partitions of $N$ into $M$ summands is the same as the number of non-negative partitions of $N - M$ into $M$ summands, which is the binomial coefficient $\binom{N-1} {N-M} = \binom{N-1}{M-1}$.
\end{proof}

\paragraph{5.1.2} Now let's consider ordered partitions of $N$ into $M$, $M+1$, or $M+2$ summands, with $M$ specific summands even.

\begin{lem} \label{l:5.2}
	(i): The number of non-negative partitions of $N$ into $M$ even summands is the coefficient of $X^N$ in the expansion $(1+x^2)^{N+M-1}$, which is also the coefficient of $x^N$ in $(1 + x^2 + x^4 + \cdots )^M$. \vspace{3pt}
	This is the same mod 2 as the coefficient of $x^N$ in $(1 + x + x^2 + \cdots)^{2M}$, which is $\binom{N + 2M -1}{N} = \binom{N+2M-1}{2M-1}$, by Lemma~\ref{l:5.1}(i).\\ 

	(ii): The number mod 2 of positive partitions of $N$ into $M$ even summands is $\binom{N - 1} {N - 2M} = \binom{N - 1}{2M - 1}.$\\

	(iii): The number mod 2 of positive partitions of $N$ into $M+1$ summands, all but the first even (or all but the last even) is $\binom{N-1}{N - (2M+1)} = \binom{N-1}{2M}.$ \\
	
	(iv): The number mod 2 of positive partitions of $N$ into $M+2$ summands, all but the first and last even, is $\binom{N-1}{N-(2M+2)} = \binom{N-1}{2M+1}$.
\end{lem}

\begin{proof}
	(i) and (ii): The first statement is essentially the same as Lemma~\ref{l:5.1}(i).
	We then count mod 2, and use the fact that $(1+ x^2 + x^4 + \cdots )^M \equiv (1 + x + x^2 + ...)^{2M}$ mod 2.
	To count positive even partitions, we add 2 to each non-negative term of a partition of $N - 2M$ into $M$ even pieces.\\ 
	
	(iii): First we count the number mod 2 of non-negative partitions of $N$ into $M+1$ terms where all terms except the first term, or all terms except the last term, are even.
	The count is the coefficient of $x^N$ in
	$$(1 + x + x^2 + \cdots)(1 +x^2 + x^4 + \cdots)^M \equiv (1 +x + x^2 + \cdots)^{2M+1}.$$
	The answer from Lemma~\ref{l:5.1} is $\binom{N+2M}{N}$.
	To get positive such partitions, subtract $2M+1$ from $N$.\\
	
	(iv): The mod 2 arithmetic first takes us to the coefficient of $x^N$ in
	$$(1+x + x^2 +\ldots)(1+x^2 + x^4 + \cdots)^M(1 + x + x^2 + \cdots) \equiv (1 +x + x^2 + \cdots)^{2M+2}$$
	to count non-negative such partitions.
	The answer is $\binom{N+2M+1}{N}$.
	To get positive such partitions subtract $2M+2$ from $N$, which means subtract 2 for each even piece and 1 for the other two pieces.
\end{proof}

\subsection{Proofs of Theorems \ref{t:binomial repeated} and \ref{t:2.2}}

\paragraph{5.2.1} We now deduce Combinatorial Fact 1 stated in \S2.3.2.
For odd $n = 2m + 1$, we first want to partition $N = j+1$ into $M = m+1$ positive even terms.
The count, explained in Lemma~\ref{l:5.2}(ii) above, is $\binom{j}{2m+1} = \binom{j}{n}$.\\

We have also covered the other part of Combinatorial Fact 1, where $N = i+1$ is partitioned into $m+2$ positive pieces, all but the first and last even. 
The answer from Lemma~\ref{l:5.2}(iv) is $\binom{i}{i-(2m+1)} = \binom{i}{n}$.\\

One gets in the same way the claims of Combinatorial Fact 2, when $n = 2m$ is even.
We partition $N = i+1$ into $m+1$ positive pieces, all but the first even.
The count from Lemma~\ref{l:5.2}(iii) is $\binom{i}{i-2m} = \binom{i}{n}$.\\  

The combinatorial facts just established imply Theorem~\ref{t:binomial repeated}.

\paragraph{5.2.2} Finally we prove Theorem~\ref{t:2.2}.
In the computation for $E\g1_2$, with $|I_1|$ even, we need when $n = 2m$ the number mod 2 of partitions of $N = i+1$ into $m+1$ positive even summands.
By Lemma~\ref{l:5.2}(ii) above, this is $\binom{i}{2m+1} = \binom{i}{n+1}$.
When $n = 2m+1$ and $|I_1|$ is even we need the number mod 2 of partitions of $N = i+1$ into $m+2$ positive summands, all but the last even.
This is also calculated in Lemma~\ref{l:5.2}(iii) as $\binom{i}{2m+2} = \binom{i}{n+1}$.
When $|I_1|$ is arbitrary, the previous counts were $\binom{i}{n}$ in both the $n$ even and $n$ odd cases.
This means the counts with $|I_1|$ odd is the sum mod 2, $\binom{i}{n+1} + \binom{i}{n} = \binom{i+1}{n+1}$, which implies the statements in Theorem~\ref{t:2.2}.

\subsection{Manipulations With the Relations}

\paragraph{5.3.1} We first want to use the relations in Theorem~\ref{Adem2} to prove the Adem relations in their usual form, expressing an inadmissible composition $Sq^aSq^b, a < 2b$, as a sum of admissible compositions,
$$Sq^aSq^b = \sum_i \binom{b-1-i}{a-2i}Sq^{a+b-i}Sq^i.$$
Of course we want cochain level versions of these relations, writing differences of specific cocycles as coboundaries.\\

To keep the notation here consistent with the statement of Theorem~\ref{Adem2}, we will grade cocycles in positive degrees.
In \S2.1.4, we pointed out that on the cocycle level $s(\alpha \smallsmile_i \alpha) = s\alpha \smallsmile_{i+1} s\alpha$, where $s$ is cochain suspension.
This is a strong form at the cocycle level of the commutativity of Squares with suspension.
It then suffices to prove the standard Adem relation on cocycles of very high degree $n$, since, by desuspending, it will also hold on cocycles of lower degree.
In fact, a specific cochain level relation in high degree desuspends to a cochain level relation in lower degrees.
It is unclear how such desuspended relations compare to the unstable relations of Theorem~\ref{Adem2}, although there do exist simply described algorithms for computing such desuspensions in terms of Surj operations or coface operations.\\

First, fix $a$ and $b$.
With $m$ large, we set $n = 2^m-1 + b,\ p = 2^m-1,\ q = 2n -a$.\footnote{We found these values of $n,p,q$ in some lecture notes of J.
Lurie, \cite{13lurie}, although they probably go back to Adem.}
Thus $a = 2n -q,\ b = n-p$.
Suppressing the cocycle $\alpha$ of degree $n$, the cohomology relation in Theorem~\ref{Adem2} is
\begin{equation*}
\sum_\ell \binom{q-\ell}{q-2\ell} Sq^{2n-p-\ell} Sq^{n-q+\ell} \ +\ \sum_\ell \binom{p-\ell}{p-2\ell} Sq^{2n-q-\ell} Sq^{n-p+\ell} \ =\ 0.
\end{equation*}
In the right-half sum, $\ell = 0$ gives $Sq^aSq^b$.
Since $p = 2^m -1$, all its base 2 expansion coefficients are 1's.
Recall that if $A = \sum a_i2^i$ and $B = \sum b_i2^i$, with $a_i, b_i \in \{0, 1\}$, then $\binom{A}{B} = \prod \binom{a_i}{b_i} \pmod{2}$.
Thus, for $\ell> 0$ all the binomial coefficients $\binom{p-\ell}{p-2\ell} \equiv 0 \pmod{2}$, since if $2^i$ is the greatest power of 2 dividing $\ell$ then one will see $\binom{0}{1}$ in the $i^{th}$ factor of the formula for $\binom{p-\ell}{p-2\ell}$.\\

In the left-half sum, set $i = n - q + \ell$.
Then the Square terms are $Sq^{a+b-i}Sq^i$.
The corresponding binomial coefficients are
\begin{align*}
\binom{q-\ell}{q-2\ell} & = \binom{q-\ell} {\ell} = \binom{n-i} {\ell} = \binom{n-i} {q-n+i} \\ & = \binom{n-i}{2n-q-2i} = \binom{n-i}{a-2i} = \binom{2^m +b -1 -i }{a-2i} \\ & \equiv \binom{b-1-i}{a-2i} \pmod2.
\end{align*}
The last congruence holds because $m$ is large, so the $2^m$ just adds an irrelevant 1 in the base 2 expansion.
This proves the usual Adem relation formula.\\

One can also show directly that linear combinations of the relations in Theorem~\ref{Adem2} for any fixed cocycle degree express inadmissible compositions $Sq^aSq^b(\alpha)$ as sums of admissible compositions, plus specific coboundaries.
But that argument is recursive and does not lend itself to determining Adem's general binomial coefficient formulae.
 
\paragraph{5.3.2} We want to make one more point about the method of \S5.3.1 that finds the exact Adem relations for inadmissible compositions $Sq^aSq^b$ by stabilizing the cocycle dimension and using commutativity of Squares with suspension.
 We fixed $a, b$ and chose $n = 2^m-1+b,\ p = 2^m - 1,\ q = 2n - a$.
The integer $m$ can vary.
We then examined the relation of Theorem~\ref{Adem2} for these choices of $(n, p, q)$.
But Theorem~\ref{Adem2} contains more than cohomology information.\footnote{This was the whole point of the paper!} Specifically, Theorem~\ref{Adem2} is a cocycle formula with coboundary terms.\\

Among the coboundary terms are $d(N_{q,p,n}(\alpha))$, affiliated with the right-half Square terms in Theorem~\ref{Adem2}, and $d(N_{p,q,n}(\alpha))$, affiliated with the left-half Square terms.
The affiliations are from Theorem~\ref{Adem1}.
The point of \S5.3.1 was that some binomial arithmetic showed that both halves of the Square terms in Theorem~\ref{Adem2} did not change as $m$ increased.
Specifically, these Square terms gave the usual Adem relations for $Sq^aSq^b$.
Our new point here is that also the expressions for $N_{q,p,n}(\alpha)$ and $N_{p,q,n}(\alpha)$, given in Theorem~\ref{Adem1}, are stable under cochain suspension, as $m$ varies.
First consider
$$N_{q,\, p,\, n}(\alpha) \ =\!\! \sum_{\substack{0< a\le \ell \in \Zee[1/2];\ \\ a\equiv \ell\ mod\ \Zee}}\binom{p-\ell-a}{p-2\ell} \binom{p- \ell +a}{p-2\ell}Sq^{n -p+\ell +a}(\alpha)\ \smallsmile_{q-p+2\ell+1}\ Sq^{n -p + \ell -a}(\alpha).$$

If $m$ increases by 1, then $n, p$ and $q-p$ increase by $2^m$, and $q$ increases by $2^{m+1}$.
So $n - p$ is unchanged.
Set $M = 2^m$.
Then
$$N_{q+2M,\, p+M,\, n}(s^M\alpha) = s^MN_{q,\, p,\, n}(\alpha).$$ 
The proof involves binomial arithmetic, similar to what was done in \S5.3.1, and also the cochain suspension formula for $\smallsmile_r$'s mentioned in \S2.1.4 that reads $sx \smallsmile_{r+1} sy = s(x \smallsmile_r y).$ The binomial arithmetic is easy and uses that if $p = 2^m -1$ and $0 < a < \ell \leq p/2$ then $$\binom{2^m+p-\ell -a}{2^m +p-2\ell}\binom{2^m +p-\ell+a}{2^m +p-2\ell} \equiv \binom{p-\ell-a}{p-2\ell} \binom{p-\ell +a}{p-2\ell} \pmod 2,$$ since the $2^m$'s just add irrelevant 1's on the left end of the base two representations of all the other numbers occurring in the binomial coefficients.\\
 
We also get
$$N_{p+M,\, q+2M,\, n}(s^M\alpha) = s^MN_{p,\, q,\, n}(\alpha).$$
This is trickier.
Note increasing $m$ by 1 decreases both $n-q$ and $p-q$ by $M = 2^m$.
Also, the binomial arithmetic is a bit trickier because in analyzing the summation with binomial coefficients one needs to work with a new variable $j$ with $\ell = M + j$.
This also takes care of the $\smallsmile_{p-q + 2\ell +1}$ product of two Square terms, so that the cochain suspension formula for $\smallsmile_r$'s still works out.\\

It is unclear, to put it mildly, how the other coboundary terms in Theorem~\ref{Adem2} behave under these same $2^m$-fold cochain suspensions.
It seems that it would be quite a nice result if a priori, preferred coboundary formulae for Adem relations could be found that are compatible under cochain suspension.
This would allow constructions of simplicial set three-stage Postnikov tower {\it spectra} that would seem interesting, extending the (easy) constructions of  simplicial set, suspension compatible, two-stage Postnikov tower spectra.
The first serious example would be the 2-type of the sphere spectrum, with homotopy groups $\Zee, \Zee/2, \Zee/2$ in degrees 0, 1, 2.

\nocite{24steenrodreduced, 26steenrodepstein}
\bibliographystyle{abbrv}
\bibliography{bibliography}

\end{document}